\documentclass[11pt]{article}
\usepackage{amsthm}
\usepackage{amsmath}
\usepackage{amssymb}
\usepackage{enumerate}
\usepackage{epsfig}
\usepackage{paralist}
\usepackage{hyperref}

\usepackage[dvipsnames,usenames]{color}

\newcounter{contador}
\newtheorem{propo}[contador]{Proposition}
\newtheorem{teo}[contador]{Theorem}
\newtheorem{lem}[contador]{Lemma}
\newtheorem{nota}[contador]{Remark}

\newtheorem{corol}[contador]{Corollary}

\newcommand{\rec}{\noindent} 

\newcommand{\dps}{\displaystyle} 

\newcommand{\bx}{{\mathbf x}}
\newcommand{\by}{{\mathbf y}}

\newcommand{\K}{{\mathbb K}}
\newcommand{\R}{{\mathbb R}}
\newcommand{\C}{{\mathbb C}}
\newcommand{\N}{{\mathbb N}}

\title{Global periodicity conditions for maps and recurrences\\ via Normal Forms}
\author{Anna Cima$^{(1)}$, Armengol Gasull$^{(1)}$ and V\'{\i}ctor Ma\~{n}osa $^{(2)}$
\\*[.1truecm]
{\small \textsl{$^{(1)}$ Dept. de Matem\`{a}tiques, Facultat de Ci\`{e}ncies,}}
\\*[-.25truecm] {\small \textsl{Universitat Aut\`{o}noma de Barcelona,}}
\\*[-.25truecm] {\small \textsl{08193 Bellaterra, Barcelona, Spain}}
\\*[-.25truecm] {\small \textsl{cima@mat.uab.cat, gasull@mat.uab.cat}}
\\*[-.25truecm]
\\*[-.25truecm] {\small \textsl{$^{(2)}$ Dept. de Matem\`{a}tica Aplicada III (MA3),}}
\\*[-.25truecm] {\small \textsl{Control, Dynamics and Applications Group (CoDALab)}}
\\*[-.25truecm] {\small \textsl{Universitat Polit\`{e}cnica de Catalunya (UPC)}}
\\*[-.25truecm] {\small \textsl{Colom 1, 08222 Terrassa, Spain}}
\\*[-.25truecm] {\small \textsl{victor.manosa@upc.edu}}}

\begin{document}
\maketitle


\begin{abstract}
We face the problem of characterizing the  periodic cases in parametric families
of  rational diffeomorphisms of $\K^k$, where $\K$ is $\R$ or $\C$, having a fixed
point. Our approach relies   on the Normal Form Theory, to obtain necessary
conditions for the existence of a formal linearization of the map, and on the
introduction of a suitable rational parametrization of the parameters of the
family. Using these tools we  can find a finite set of  values  $p$ for which the
map can be $p$-periodic, reducing the problem of finding the parameters for which
the  periodic cases appear to  simple computations. We apply our results to
several two and three dimensional classes of polynomial or rational maps. In
particular we find the global periodic cases for several Lyness type recurrences.
\end{abstract}


\rec {\sl 2000 Mathematics Subject Classification:} \texttt{37G05, 39A11, 39A20,
37C05}

\rec {\sl Keywords:} Periodic maps; Linearization; Normal Forms; Rational
parametrizations; Globally periodic recurrences; Lyness recurrences.\newline

\vfill

\section{Introduction}

A map $F$ such that $F^p(x)\equiv x, $ for some $p\in\N$ and for all $x$  for which $F^p$ is well
defined,  will be called a {\it  periodic
map}. If $p$ is the smallest positive integer with this property, then $F$ is called {\it $p$-periodic}.
 In this paper we treat the problem of characterizing the $p$-periodic cases
in parametric families of rational maps of $\K^k$, where $\K$ is
$\R$ or $\C$, having a fixed point.

When $F$ is a $p$-periodic differentiable map having a fixed point,
$\bx_0$, it is well-known that $(DF(\bx_0))^p=\operatorname{Id}$. In
fact this is a simple consequence of the chain rule. As we will see
in Proposition~\ref{lineal}, $m=p$ is the smallest positive integer
number such that $(DF(\bx_0))^m=\operatorname{Id}$. This simple
result allows to treat in a easy way the periodicity problem  when a
value $p$ such that $(DF(\bx_0))^p=\operatorname{Id}$ is known. For
instance if $F$ has a fixed point $\bx_0$ such that
$(DF(\bx_0))^2=\operatorname{Id}$ then if $F$ is $p$-periodic then
$p$ must be $2,$ and not $p=2m,\,m\in\N$ as we could think in
principle,
 and then we simply have to check whether $F^2=\operatorname{Id}$ or not.

In general, given a parametric family of maps $F_{\mathbf a}$,
$\mathbf{a}\in\K^m$, the most difficult problem for finding the
periodic maps is to determine which are the possible values $p$ such
that there exists some $\mathbf{a}$ such that  $F_{\mathbf a}$ is
$p$-periodic. The tools that we will introduce in this paper will
allow to find a finite set of possible values of $p$ for which the
map can be $p$-periodic, converting the problem of finding these
values of $\mathbf a$
 into a computational problem.

Proposition~\ref{lineal} as well as our approach to the
characterization of $p$-periodic maps via Normal Form Theory are
based on the   Montgomery-Bochner Theorem, see~\cite{MZ}. It will be
recalled and proved in Section~\ref{mmbb}. In a few words it says
that any $p$-periodic, $\mathcal{C}^1$-map  with a fixed point is
locally conjugated with the linear map $L(\bx)=DF(\bx_0)\bx$, and so
{\it locally linearizable}. Notice that the differentiability
condition is necessary since it is well known that there are
periodic involutions (i.e. $F^2=\mathrm{Id}$) given by homemorphisms
with fixed points which are not linearizable, see~\cite{B}.

Hence any  $p$-periodic case in a given family with fixed points
can be locally linearized. Thus, the application of  a suitable Normal Form
algorithm, will  give necessary  conditions for  the existence of the
linearization. As we will see, these conditions are  sometimes also sufficient.

We remark that this approach does not cover the problem in its full
generality, because there are periodic
diffeomorphisms without fixed points in $\R^k$ with $k\geq 7$, see \cite{HKMS,K}.

It is well-known that  the Normal Form algorithms often lead to very complicated
expressions which are difficult to handle  when dealing with the given parameters
of the map. Sometimes, these obstructions can be significatively softened by
introducing new parameters \emph{rationally depending} on the \emph{old} ones, and
such that the coordinates of the fixed points as well as the eigenvalues of the
jacobian matrix at these fixed points,  depend rationally on these \emph{new}
parameters. This is the second main characteristic of our approach, when dealing
with concrete applications.

The Normal Form Theory is briefly recalled in Section~\ref{norderpc}. In Section
\ref{planar} we obtain some results for planar maps in the case that the linear
part of $F$ at the fixed point  is given by a matrix $\mathrm{diag}(\alpha,\beta)$
with $\alpha\beta=1$,  or $\mathrm{diag}(\alpha,1)$. As first applications of the
method, we get:

\begin{teo}\label{aplgene} Consider a smooth  complex map of the
form \begin{equation}\label{Fdeteo1} F(x,y)=\left(\alpha x\,+\sum_{i+j\geq 2}
f_{i,j}x^iy^j, \frac{1}{\alpha}\,y\,+\sum_{i+j\geq 2}
g_{i,j}x^iy^j\right),\end{equation} where $\alpha$ is a primitive $p$-root of
unity, $p\ge5$. Then the conditions
$\mathcal{P}_1(F)=\mathcal{P}_2(F)=\mathcal{P}_3(F)=0$ are necessary for $F$ to be
$p$-periodic, where
\begin{align*}
\mathcal{P}_1(F):=& \left( f_{2,1}+f_{1,1}g_{1,1} \right) {\alpha}^{4}-f_{1,1}
 \left( 2f_{2,0}-g_{1,1} \right) {\alpha}^{3}+ \left(
2g_{2,0}f_{0,2}-f_{1,1}f_{2,0}+f_{1,1}g_{1,1} \right) { \alpha}^{2}\\
&-\left( f_{2,1}+f_{1,1}f_{2,0} \right) \alpha+ f_{1,1}f_{2,0},\\
\mathcal{P}_2(F):=&g_{0,2}g_{1,1}{\alpha}^{4}- \left( g_{1,2}+g_{0,2}g_{1,1}
\right) {\alpha}^{3}+
\left( f_{1,1}g_{1,1}+2g_{2,0} f_{0,2}-g_{0,2}g_{1,1} \right) {\alpha}^{2}\\
&+g_{1,1}
 \left( -2g_{0,2}+f_{1,1} \right) \alpha+f_{1,1}g_{1,1}+
g_{1,2},
\end{align*}
and $\mathcal{P}_3(F)$ is given in Appendix A.
\end{teo}

In fact, conditions $\mathcal{P}_1(F)=0$ and $\mathcal{P}_2(F)=0$ also work for
$p=4$.

 \begin{teo}\label{aplgene2} Consider a smooth  complex map of
the form \begin{equation}\label{Fdeteo2}  F(x,y)=\left(\alpha x\,+\sum_{i+j\geq 2}
f_{i,j}x^iy^j, y\,+\sum_{i+j\geq 2} g_{i,j}x^iy^j\right),\end{equation} where
$\alpha$ is a primitive $p$-root of unity. Then the following are necessary
conditions for $F$ to be $p$-periodic:
\begin{align*}
\mathcal{P}_1(F):=& f_{1,1}=0,\\
\mathcal{P}_2(F):=& g_{0,2}=0,\\
\mathcal{P}_3(F):=& f_{1,2} \alpha-2 f_{2,0} f_{0,2}+2 f_{0,2}  g_{1,1}-f_{1,2}=0,\\
\mathcal{P}_4(F):=& g_{0,3} \alpha-g_{0,3}-f_{0,2} g_{1,1}=0,\\
\mathcal{P}_5(F):=& f_{1,3} {\alpha}^{2}+ \left( -2 f_{2,0} f_{0,3}+3 f_{0,3}
g_{1,1}+2 g_{1,2} f_{0,2}-2 f_{2,1} f_{0,2}-2
f_{1,3} \right) \alpha+ f_{1,3}+2 f_{2,0} f_{0,3}\\
&+2 f_{2,1} f_{0,2}-4  g_{2,0} f^2_{0,2}-2 g_{1,2} f_{0,2}-3 f_{0,3}
g_{1,1}=0,
\\
\mathcal{P}_6(F):=& g_{0,4} {\alpha}^{2}- \left( f_{0,3} g_{1,1}+2 g_{0,4}+g_{1,2}
f_{0,2} \right) \alpha+ g_{2,0} f^2_{0,2}+f_{0,3} g_{1,1}+g_{0,4}+ g_{1,2}
f_{0,2}=0.
\end{align*}
\end{teo}
In this last case, and in contrast with the one treated in Theorem \ref{aplgene},
it is not difficult to obtain additional periodicity conditions. Two more
periodicity conditions are given in Appendix B.

The above results are applied in several contexts. The first
application is for polynomial maps. Periodic \emph{polynomial} maps
are notorious examples of invertible polynomial ones, which, in
turn, are the focus of many deep open problems like the Jacobian
conjecture, or the linearization conjecture. This second conjecture
says that if $F:\mathbb{C}^n\to\mathbb{C}^n$ is a
 $p$-periodic polynomial map,  then there exists a polynomial automorphism
 $\varphi$ (i.e. an invertible polynomial map with polynomial inverse)
 such that $\varphi\circ F\circ \varphi^{-1}$ is a linear map.
This conjecture is true for $n=2$ and as far as we know it is open
for $n\geq 3$, see \cite[Chaps. 8 and 9]{E} and \cite{Mau}.

In Section~\ref{polynomialmaps} we characterize the $p$-periodic
maps in a family of triangular maps, see Theorem~\ref{trian}, and we
give a simple and self-contained proof of the linearization
conjecture for this case. As an application of this result and
Theorem~\ref{aplgene}  we prove:

\begin{propo}\label{cascubic} Consider a complex polynomial map
\begin{equation}\label{primeraP} F(x,y)=\left(\alpha
x\,+\sum_{i+j=2}^3 f_{i,j}x^iy^j, y/\alpha\,+\sum_{i+j=2}^3
g_{i,j}x^iy^j\right),\end{equation}  The map is $p$-periodic if and
only if $\alpha$ is a primitive $p$-root of the unity, and  it holds
one of the following conditions
\begin{enumerate}[(i)]
\item $p=1$ and $F(x,y)=(x,y)$;
   \item $p=2,4$ and $F(x,y)=(\alpha x+f_{0,2}y^2,y/\alpha)$ or $F(x,y)=(\alpha
   x,y/\alpha+g_{2,0}x^2)$;
   \item $p=3$,  and $F(x,y)=(\alpha x+f_{0,3}y^3,y/\alpha)$ or $F(x,y)=(\alpha
   x,y/\alpha+g_{3,0}x^3)$;
  \item  $p\ge5$ and $F(x,y)=(\alpha x+f_{0,2}y^2+f_{0,3}y^3,y/\alpha)$ or $F(x,y)=(\alpha
   x,y/\alpha+g_{2,0}x^2+g_{3,0}x^3)$.
\end{enumerate}
\end{propo}

Similarly, as an application of Theorem~\ref{aplgene2}, we prove:

\begin{propo}\label{propoultima}
The only $p$-periodic cases in the family of complex maps
$$
F(x,y)=\left(\frac{\alpha x+b{x}^{2}+cxy+d{y}^{2}}{1+m \left(
{x}^{2} +{y}^{2} \right)},\frac {y+r{x}^{2}+sxy+t{y}^{2}}{1+m \left(
{x}^{2 }+{y}^{2} \right)}\right),
$$ are, either  $F(x,y)=(x,y)$ when $\alpha=1$, or the ones given the polynomial maps
$ F(x,y)=(\alpha x+dy^2,y)$ or $F(x,y)=(\alpha x,y+rx^2)$ when
$\alpha$ a primitive $p$-root of the unity with $p>1$ and $d$ and
$r$ arbitrary complex numbers.
\end{propo}

In all the rest of  examples, given in Sections \ref{seceed} and \ref{r3}, the
maps are the ones associated to some recurrences. Recall that given a recurrence,
autonomous or not,  it is said that it is {\it globally $p$-periodic} if for all
initial conditions for which the sequence is well-defined it gives rise to a
$p$-periodic sequence and $p$ is the smallest positive integer number with this
property. We will face this question studying an associated map $F$. With this
point of view, the recurrence will be globally periodic if and only if the map $F$
is periodic.

The study of the global periodicity in difference equations is nowadays the
subject of an active research, see for instance
\cite{Abu1,Abu,BL06,BL07,BLiz,BS,CL09,CL10,CGTopa,CGTopa3,CL,GL,M,R,RM}, and
references therein and several techniques have been used to approach the problem.
To the best of our knowledge, this is the first time that the Normal Form Theory
is used in this setting. As a second application of Theorem \ref{aplgene}, we
classify the globally periodic second order Lyness recurrences, reobtaining the
results in \cite{CL} for this case:

\begin{propo}\label{lynesnormal}
The only globally periodic Lyness recurrences
$x_{n+2}=\dps{\frac{a+x_{n+1}}{x_n}}$ with $a\in\C$, are the $5$-periodic case
with $a=1$; and the $6$-periodic case with $a=0$.
\end{propo}

Also as a direct consequence of Theorem \ref{aplgene} we get next
result for some Gumovski-Mira-type recurrences~\cite{GM},
\begin{propo}\label{gummir}
There are no globally periodic cases in the family of Gumovski-Mira
recurrences
$$x_{n+2}=-x_{n}+\frac{ x_{n+1}}{b+x_{n+1}^2},\quad b\in\C.$$
\end{propo}

One of the main applications in this setting  concerns the $2$-periodic Lyness
recurrence
\begin{equation}\label{eq}
x_{n+2}\,=\,\frac{a_n+x_{n+1}}{x_n},\, \quad\mbox{ where }\quad
a_n\,=\,\left\{\begin{array}{lllr} a&\mbox{for}&n=2\ell+1,\\
b&\mbox{for}&\,n=2\ell,
\end{array}\right.\end{equation}
and $a,b\in\C$. In Section \ref{secGPFba} we solve the global
periodicity problem for it by studying the family of maps
\begin{equation*}
F_{b,a}(x,y)=\left(\frac {a+y}{x},\frac {a+b x+ y}{x y}\right),
\end{equation*}
which as we will see describes the behavior  of~\eqref{eq}.

\begin{teo}\label{teolyness2p}
The only globally periodic recurrences in~\eqref{eq} are:
\begin{enumerate}[(i)]
  \item The cases  $a=b=0$ ($6$-periodic) and  $a=b=1$ ($5$-periodic).
 \item The cases $a=(-1\pm i\,\sqrt{3})/{2}$ and  $b=\overline {a}=1/a$, $10$-periodic.
   \end{enumerate}
\end{teo}

Notice that the   cases given in (i)  correspond to the  well-known autonomous
globally periodic Lyness recurrences also appearing in
Proposition~\ref{lynesnormal}.

Finally, to show an application in $\K^3$  we find the globally
periodic third order Lyness recurrences, reobtaining again the
result in \cite{CL}:

\begin{propo}\label{lynes3d}
The only globally periodic third-order Lyness recurrence
$$x_{n+3}=\dps{\frac{a+x_{n+1}+x_{n+2}}{x_n}},\quad a\in\C,$$  corresponds to $a=1$
and is  $8$-periodic.
\end{propo}

\section{Some consequences of the Montgomery-Bochner Theorem}\label{mmbb}

The next version of Montgomery-Bochner Theorem is a simplified one,
adapted to our interests. The general one applies in a much more
general context, see~\cite{MZ}.

\begin{teo}[{\bf  Montgomery-Bochner}]
Let $F:\mathcal{U}\to\mathcal{U}$  be a  $p$-periodic
$\mathcal{C}^1$-diffeo\-mor\-phism, where $U$ is an open set  of $\K^k$. Let
$\bx_0\in\mathcal{U}$ be  a fixed point of $F$. Then, there exists a neighbourhood
of $\bx_0$ where $F$ is conjugated with the linear map $L(\bx)=DF(\bx_0)\bx$.
Moreover  the linearization is given by the local diffeomorphism
  $$
\psi(\bx)=\frac{1}{p}\sum_{i=0}^{p-1}
(DF(\bx_0))^{-i}\left(F^i(\bx)\right).
  $$
\end{teo}

\begin{proof}  Since $F$ is $p$-periodic $(DF(\bx_0))^p=\operatorname{Id}$.
So $(\det(DF(\bx_0))^p=1$ and $DF(\bx_0)$ is invertible. Consider $\psi$ as in the
statement. By the inverse function theorem it is clear that the map $\psi$ is a
local diffeomophism because $D\psi(\bx_0)=\operatorname{Id}$. Moreover, using again the $p$-periodicity of $F$ we get
that $\psi(F(\bx))=L(\psi(\bx))$, as we wanted to prove.
\end{proof}

As we have seen in the proof of the above theorem, if  $F$ is a $p$-periodic
differentiable map with a fixed point $\bx_0,$ then
$(DF(\bx_0))^p=\operatorname{Id}$. Next result relates $p$ with the minimum
positive $m$ such that $(DF(\bx_0))^m=\operatorname{Id}$.

\begin{propo}\label{lineal} Let $F$ be a differentiable map having a fixed point $\bx_0$.
Assume that $F$ is $p$-periodic and let $m$ be the minimum positive
$m$ such that $(DF(\bx_0))^m=\operatorname{Id}$. Then $p=m.$
\end{propo}
\begin{proof} By using the  Montgomery-Bochner Theorem we know that $F$ is
$\mathcal{C}^1$-conjugated to  $L(\bx)=DF(\bx_0)\bx$ in a
neighborhood of $\bx_0$. Thus $F=\psi^{-1}\circ L\circ \psi$, for
some $\mathcal{C}^1$ diffeomorphism $\psi$. Since
$L^m=\operatorname{Id}$ if and only if $F^m= \psi^{-1}\circ L^m\circ
\psi=\psi^{-1}\circ \psi=\operatorname{Id},$ the result follows.
\end{proof}

\begin{corol}\label{pbaix} Let $F_{\mathbf a}(\bx)=L\bx+G(\bx,\mathbf{a}),$ with
$\bx\in\mathcal{U}\subset\K^n$ and  $\mathbf{a}\in\K^m,$ a smooth
family of maps such that $G({\bf 0},\mathbf{a})\equiv D_{\bf
x}G({\bf 0},\mathbf{a})\equiv0$ for all $\mathbf{a}\in\K^m$. Assume
that $p$ is the minimum positive integer number such that
$L^p=\operatorname{Id}$. Then if $F_{\mathbf a}$ is periodic for
some $\mathbf{a}\in\K$ then it is $p$-periodic, i.e.
$F_{\mathbf{a}}^p=\operatorname{Id}.$
\end{corol}

In  particular note that  if $L=\operatorname{Id}$ then the only periodic case is
$F_\mathbf{a}(\bx)=\bx$ and when $L^2=\operatorname{Id}$ the periodicity conditions
are given by $F_\mathbf{a}(F_\mathbf{a}(\bx))\equiv\bx.$ For example the  fact proved in
\cite[Ex. 2]{R}, that the only periodic map of the form $F(x_1,x_2)=(x_2+a
x_1^2,x_1+bx_1x_2)$ corresponds to the linear case $a=b=0$, follows easily using
this approach.

Notice that using Montgomery-Bochner Theorem  a necessary condition
for a map of the form $F_\mathbf{a}(\bx)=L\bx+G(\bx,\mathbf{a}),$ to
be periodic is that $F_\mathbf{a}$ is linearizable in a
neighbourhood of~$\bf 0$. The linearizable cases can be detected by
following the well-know Normal Form Theory, which, as far as we
know, has not been used for this purpose. Some results useful for
applying it will be recalled in the next section.

\section{Periodicity conditions via Normal Form Theory}\label{norderpc}

We start introducing some well-known issues of Normal Form Theory, while referring
the reader to \cite[Sec. 2.5]{AP}, for further details.

Let $F:=F^{(1)}:\K^k\rightarrow\K^k$,  be a family of smooth maps
depending on some parameters and satisfying $F^{(1)}(0)=0.$ Let
\begin{equation}\label{F1}
F^{(1)}(\bx)=F^{(1)}_1(\bx)+F^{(1)}_2(\bx)+\cdots
+F^{(1)}_k(\bx)+O(|\bx|^{k+1})
\end{equation}
be the Taylor expansion of $F$ at $\mathbf{0},$ where $F_r^{(1)}\in
\mathcal{H}_r,$ the real vector space of maps whose components are
homogeneous polynomials of degree $r.$

The aim of the Normal Form Theory is to construct a sequence of transformations
$\Phi_n$, starting from $n=2$, such that at each step, $\Phi_n$  simplifies, as
much as possible, the terms of the corresponding homogeneous part of degree $n$.
To this end, let $F_1^{(1)}(\bx)=DF^{(1)}(\mathbf{0})\,\bx=:L\bx$ and suppose that
$$
F^{(n-1)}(\bx)=L \,\bx+F^{(n-1)}_n(\bx)+O(|\bx|^{n+1})\,,\,n\ge 2.
$$
Consider a transformation
$$
\bx=\Phi_{n}(\by):=\by+\phi_n(\by),
$$
with $\phi_n\in\mathcal{H}_n,$ such that it conjugates the map
$F^{(n-1)}$ with a new map $F^{(n)}$, via the conjugation
$$
F^{(n-1)}(\Phi_n)=\Phi_n(F^{(n)}).
$$
From the above equation, it can be  easily seen that
$$
F^{(n)}(\by)=L\,\by+L\,\phi_n(\by)-\phi_n(L
\by)+F_n^{(n-1)}(\by)+O(|\by|^{n+1}).
$$
Clearly, if $\phi_n(\by)$ can be chosen in such a way that
\begin{equation}\label{homological}
M_L\left(\phi_n(\by)\right):=L\,\phi_n(\by)-\phi_n(L \by)=-F_n^{(n-1)}(\by),
\end{equation}
then $F^{(n-1)}$ is transformed into
$$F^{(n)}(\by)=L \,\by+F^{(n)}_{n+1}(\by)+O(|\by|^{n+2})=L\,\by+O(|\by|^{n+1}).$$

The vectorial equation (\ref{homological}) is the well-known \emph{homological
equation} associated with $L=DF^{(1)}(\mathbf{0})$, and the existence of solutions
of it is the  necessary  and sufficient condition  to be able to remove the
homogeneous terms of degree $n$.

From now, one we will assume that the linear map is diagonalizable,
and so that it is $L=\mathrm{diag}(\lambda_i)_{i=1}^k$. In this
case, the linear operator
$M_L:=\mathcal{H}_n\rightarrow\mathcal{H}_n,$ given
in~\eqref{homological}, has the eigenvectors $\bx^{\mathbf m}\,e_i,$
$i=1,2,\ldots,k$, with ${\mathbf
m}=(m_1,m_2,\ldots,m_k)\in\mathcal{M}_n^k:=\{{\mathbf m}\in\N^k$
satisfying $\sum_{i=1}^k m_i=n\}$;
 $\bx^{\mathbf
m}=x_1^{m_1}\,x_2^{m_2}\cdots\,x_n^{m_k}$ where $\bx=(x_1,x_2,\ldots,x_k)\in\K^k$;
and where $e_i$ is the $i$-th member of the natural basis for $\K^k.$ Hence
\begin{equation}\label{vepsvapsml}M_L\left(\bx^{\mathbf m}\,e_i\right)=(\lambda_i-{\mathbf
\lambda}^{{\mathbf m}})\,\bx^{\mathbf m}\,e_i,\end{equation}
 where
${\mathbf \lambda}^{{\mathbf
m}}=\lambda_1^{m_1}\,\lambda_2^{m_2}\cdots\,\lambda_n^{m_k}.$

Set $$F_{n}^{(n-1)}(\bx)=\left(\sum_{{\mathbf m}} f^{(n-1)}_{1;{\mathbf m}}
\bx^{\mathbf m},\sum_{{\mathbf m}} f^{(n-1)}_{2;{\mathbf m}} \bx^{\mathbf
m},\ldots,\sum_{{\mathbf m}} f^{(n-1)}_{k;{\mathbf m}} \bx^{\mathbf m}\right),$$
and
$$\phi_{n}(\bx)=\left(\sum_{{\mathbf m}}
a_{1;{\mathbf m}} \bx^{\mathbf m},\sum_{{\mathbf m}} a_{2;{\mathbf m}}
\bx^{\mathbf m},\ldots,\sum_{{\mathbf m}} a_{k;{\mathbf m}} \bx^{\mathbf
m}\right),$$ where ${\mathbf m}\in\mathcal{M}_n^k$.

When $\lambda_i-{\mathbf \lambda}^{{\mathbf m}}\ne 0$ for all the
suitable values of $\mathbf m\in\N^k$ and for all $i=1,2,\ldots,k$,
it is said that there are no resonances. In this case the operator
$M_L$ is invertible, the homological equation always has solution
and so the linearization process can continue. On the contrary, if
$\lambda_i-{\mathbf \lambda}^{{\mathbf m}}=0$ for some  $\mathbf
{m}\in\mathcal{M}_n^k$ and some $i\in\{1,2,\ldots,k\}$, then the
vector $\mathbf {\lambda}=(\lambda_1,\lambda_1,\ldots,\lambda_k)$ is
said to be {\it resonant of order} $n$.  In this case,  by simple
inspection of the homological equation, and using
(\ref{vepsvapsml}), we obtain that the  \emph{the $n$th order
obstruction equation associated to the resonance} is given by
$$
(\lambda_i-{\mathbf \lambda}^{{\mathbf m}})\,a_{i;{\mathbf
m}}\,=\,-f^{(n-1)}_{i;{\mathbf m}}.
$$
However, there are some maps having this resonance  for which the process can
continue. This happens if the right-hand side of this scalar equations vanish,
namely $f^{(n-1)}_{i;{\mathbf m}}=0$, and these cases are the ones candidate to be
linearized. Hence,
 we have obtained the following result
 \begin{propo}\label{novaproposiciometode}
If $L:=\mathrm{diag}(\lambda_i)_{i=1}^k$, then a necessary condition for the map
(\ref{F1}) to be  periodic is given by  the $n$th order periodicity condition
associated to the resonance condition, $\lambda_i-{\mathbf \lambda}^{{\mathbf
m}}=0$, given by $ f_{i;{\mathbf m}}^{(n-1)}=0. $
 \end{propo}

 \begin{nota}
Notice that $p$-periodic maps with $L$ diagonal are such that $\lambda^p_i=1,$ for all $i$.
Therefore for  these maps many resonances
 $\lambda_i-{\mathbf\lambda}^{\mathbf m}=0$ appear.
  \end{nota}

By following the Normal Form Algorithm, it is straightforward (and well known) to
see that the numerator of  $f_{i;{\mathbf m}}^{(n-1)}$ is a polynomial in the
coefficients of $F^{(1)}$. Thus, for each particular case, the above equations
give  periodicity conditions, which are algebraic in terms of the initial
parameters of the map, once expressed in form (\ref{F1}).

To fix the ideas we give a simple  example. Suppose that $k=2$. Assume that
$L=\mathrm{diag}(\alpha,\beta)$. Set $\Phi_{2}(\by):=\by+\phi_2(\by)$, where
$$\phi_2(x,y):=\left(\begin{array}{cc}
a_{20}x^2+a_{11}xy+a_{02}y^2\\b_{20}x^2+b_{11}xy+b_{02}y^2\end{array}\right).$$
Consider the map $F^{(1)}(\bx)=L\bx+F_2^{(1)}(\bx)+O(|\by|^{3})$
with
$$F_2^{(1)}(x,y)=\left(\begin{array}{l} f_{20}x^2+f_{11}xy+f_{02}y^2\\
g_{20}x^2+g_{11}xy+g_{02}y^2
\end{array}\right),$$
where to simplify the notation, and from now on, if there is no possibility of
confusion, we will drop the superscript ${\scriptstyle{(1)}}$ of the coefficients
of $F^{(1)}$.

 The homological
equation at order $2$ is $L\,\phi_2(\by)-\phi_2(L
\by)=-F_2^{(1)}(\by)$, and gives the following six scalar equations:
$$
\left\{
  \begin{array}{ll}
    (\alpha-\alpha^2)a_{20}=-f_{20},  &(\beta-\alpha^2)b_{20}=-g_{20},\\
    (\alpha-\alpha\beta)a_{11}=-f_{11},  & (\beta-\alpha\beta)b_{11}=-g_{11}, \\
    (\alpha-\beta^2)a_{02}=-f_{02},  & (\beta-\beta^2)b_{02}=-g_{02}.\\
  \end{array}
\right.
$$
If no one of the six \emph{$2$nd order  resonance conditions}:
$\alpha^2-\alpha,$ $\alpha \beta-\alpha,$ $\beta^2-\alpha,$
$\beta^2-\beta,$ $\alpha\beta-\beta,$ and $\alpha^2-\beta$, vanish,
there is no obstruction to remove the second order terms of
$F^{(1)}$ using the conjugation $\Phi_2$.

Suppose now, that the map $F^{(1)}$ is such that the
\emph{resonance} $\beta-\alpha^2=0$ occurs. Then the scalar equation
$(\beta-\alpha^2)b_{20}=-g_{20}$ is an \emph{obstruction equation}.
But this obstruction to the linearization process disappears if
$g_{20}$ vanishes. In summary,  if $\beta=\alpha^2$, then $g_{20}=0$
is a \emph{periodicity condition}.

\section{Proof of Theorems \ref{aplgene} and \ref{aplgene2}}\label{planar}

We keep the notation introduced in the above section, i.e., $F^{(k)}$ is the map
obtained after $k-1$ steps of the normal form procedure, $ F^{(k)}(\bx)=L
\,\bx+F^{(k)}_{k+1}(\bx)+O(|\bx|^{k+2}),$ and its coefficients are $f_{i,j}^{(k)}$
and $g_{i,j}^{(k)}$.  First consider the case treated in Theorem \ref{aplgene}:
$$
F^{(k)}(x,y)=\left(\alpha x\,+\sum_{i+j\geq k+1}
f_{i,j}^{(k)}x^iy^j,\frac{1}{\alpha}\, y\,+\sum_{i+j\geq k+1}
g_{i,j}^{(k)}x^iy^j\right).$$ It is easy to check that the scalar
equations associated to equation (\ref{homological}) are
$$
\left\{
  \begin{array}{l}
    \alpha(1-\alpha^{n-2i-1})\, a_{n-i,i}=-f^{(n-1)}_{n-i,i}, \\
    \alpha^{-1}(1-\alpha^{n-2i+1})\, b_{n-i,i}=-g^{(n-1)}_{n-i,i},
  \end{array}
\right.
$$
for $i=0,\ldots,n$. Hence, for any $\alpha$, and any odd $n$ we
obtain the periodicity conditions
\begin{equation}\label{mcc}
f^{(n-1)}_{\frac{n+1}2,\frac{n-1}2}=0 \quad\mbox{and}\quad
g^{(n-1)}_{\frac{n-1}2,\frac{n+1}2}=0.
\end{equation}
We remark that for a given  $\alpha$, primitive $m$-root of the unity, other
conditions can be added. For instance when $m=3$,
 $$
f_{0,2}^{(1)}=0,\,g_{2,0}^{(1)}=0,\,f_{4,0}^{(3)}=0,\,f_{1,3}^{(3)}=0,\,g_{3,1}^{(3)}=0
\,\mbox{ and }\, g_{0,4}^{(3)}=0,
  $$
 are also periodicity conditions. Also, when $m=5$, $f_{0,4}^{(3)}=0$
  and $g_{4,0}^{(3)}=0$ are periodicity conditions.

 Returning to the general case, we want to obtain explicitly the
first conditions given in~\eqref{mcc}. The first four  are given by
$f_{2,1}^{(2)}=g_{1,2}^{(2)}=0$; and $f_{3,2}^{(4)}=g_{2,3}^{(4)}=0$. Some
straightforward computations applying
 the Normal Form Algorithm explained in Section \ref{norderpc} show that
\[
f_{2,1}^{(2)}=\frac{\mathcal{P}_1(F)}{\alpha(\alpha^3-1)},\quad
g_{1,2}^{(2)}=\frac{\mathcal{P}_2(F)}{\alpha^3-1}\quad\mbox{and}\quad
f_{3,2}^{(4)}=\frac{\mathcal{P}_3(F)}{{\alpha}^{3} (\alpha^3-1)^3 \left(
{\alpha}^{2}+1 \right)  \left( \alpha+1 \right)
},\\
\] giving the desired result. Recall that $\alpha^n-1\ne0, n\le4$. We do
not give the periodicity condition associated to $g_{2,3}^{(4)}$ for
the sake of brevity and because we will not use in the specific
examples.

 Now we consider the maps that appear in Theorem \ref{aplgene2}. When $\alpha=1$
the result follows trivially by Corollary~\ref{pbaix}. When $\alpha\ne1$, the
scalar equations associated to equation (\ref{homological}) are
$$
\left\{
  \begin{array}{l}
    \alpha(1-\alpha^{n-i-1})\, a_{n-i,i}=-f^{(n-1)}_{n-i,i}, \\
   (1-\alpha^{n-i})\, b_{n-i,i}=-g^{(n-1)}_{n-i,i},
  \end{array}
\right.
$$
for $i=0,\ldots,n$. Hence, for any $\alpha\neq 1$, we obtain the
periodicity conditions
$$
f^{(n-1)}_{1,n-1}=0 \quad\mbox{and}\quad g^{(n-1)}_{0,n}=0.
$$
Thus the first periodicity conditions are given by $f_{1,1}=g_{0,2}=0$;
$f_{1,2}^{(2)}=g_{0,3}^{(2)}=0$; and $f_{1,3}^{(3)}=g_{0,4}^{(3)}=0$. Applying
 the Normal Form Algorithm we get
\[
f_{1,2}^{(2)}=\frac{\mathcal{P}_3(F)}{\alpha-1},\quad
g_{0,3}^{(2)}=\frac{\mathcal{P}_4(F)}{\alpha-1},\quad
f_{1,3}^{(3)}=\frac{\mathcal{P}_5(F)}{\alpha-1}\quad\mbox{and}\quad
g_{0,4}^{(3)}=\frac{\mathcal{P}_6(F)}{\alpha-1}.\\
\]

 \section{On some polynomial and rational maps}\label{polynomialmaps}

In this section, first  we study  the periodicity problem for a family of
triangular maps and then we apply this result and Theorem~\ref{aplgene} to
characterize the periodic   maps  of the family~(\ref{primeraP}), proving
Proposition~\ref{cascubic}. Finally we prove Proposition~\ref{propoultima}.

\subsection{Preliminary results and a triangular family}\label{prelcub}

First, recall  that a polynomial automorphism is a bijective
polynomial map with polynomial inverse. Also recall the following
well-known result, where as usual $\C^*$=$\C\setminus\{0\}.$

\begin{lem}\label{detconstant}
Let $F:\mathbb{C}^n\to\mathbb{C}^n$ be  polynomial map.
 \begin{itemize}
   \item[(a)] If $F$ is an automorphism then
   $\det (DF(\bx))\in \C^*.$
   \item[(b)] If $F$ is $p$-periodic, $p\ge1$ then $F$ is an
   automorphism and $\det (DF(\bx))\in\C^*.$
 \end{itemize}
\end{lem}
\begin{proof}
(a) Since $F\circ F^{-1}=\operatorname{Id}$, we have that $
\det\left(DF(F^{-1}(\mathbf{x}))\right)\cdot
\det\left(DF^{-1}(\mathbf{x})\right)\equiv 1.$ Since the only
non-vanishing complex polynomials are the constant ones, the result
follows.

(b) If $F$ is $p$-periodic, then  $ F^{-1}=F^{p-1}$, hence $F$ is a
polynomial automorphism.
\end{proof}

In fact, the reciprocal of item (a) above is precisely the celebrated Jacobian
Conjecture.

Next results  characterize the periodic cases in  a family of triangular maps.

\begin{teo}\label{trian}
Let $\alpha$ be a primitive $p$-root of unity, and consider
the $\mathcal{C}^1$-map $F:\K^2\to\K^2$,
$$
    F(x,y)=\left(\alpha x+f(y),y/\alpha\right),\label{pppp}
$$
with $f(0)=f'(0)=0.$ Then:
\begin{enumerate}[(i)]
\item $F$ is periodic if and only if
$$
\sum_{j=0}^{p-1} \alpha^jf(\alpha^j y)\equiv 0\label{ccp}
$$
and then it is $p$-periodic.
\item If $F$ is $p$-periodic, the  linearization given in the Montgomery-Bochner Theorem,
  $$
\psi(\bx)=\frac{1}{p}\sum_{j=0}^{p-1} (DF(0,0))^{-j}(F^j)(\bx),
  $$
is a global linearization.
\end{enumerate}
\end{teo}

\begin{proof}
(i) By Corollary~\ref{pbaix}, if $F$ is periodic then it is $p$-periodic. It is
not difficult to prove that
\[F^p(x,y)=\left(\alpha^px +\sum_{j=1}^{p} \alpha^{p-j} f\left(\frac{y}{\alpha^{j-1}}\right)
, \frac{y}{\alpha^p}\right).\] Therefore, using that $\alpha^p=1,$
the condition of being $p$-periodic writes as
\[\sum_{j=1}^{p} \alpha^{p-j} f\left(\frac{y}{\alpha^{j-1}}\right)=\sum_{j=1}^{p} \alpha^{p-j} f\big(\alpha^{p+1-j}y\big)=
\sum_{n=1}^{p} \alpha^{n-1} f\big(\alpha^{n}y\big)=0,
 \]
for all $y\in\K$. Multiplying the last expression by $\alpha$ we obtain
condition~\eqref{ccp}.

(ii) Note that in the expression of the local diffeomorphism $\psi$ given in the statement,
\[
(DF(0,0))^{-j}(F^j)(\bx)=\left(
                           \begin{array}{cc}
                             \alpha^{-j} & 0 \\
                             0 & \alpha^j \\
                           \end{array}
                         \right)\left(
                                  \begin{array}{c}
                                    \alpha^jx+ g_j(y) \\
                                    y/\alpha^j \\
                                  \end{array}
                                \right)= \left(
                                  \begin{array}{c}
                                    x+ \alpha^{-j}g_j(y) \\
                                    y \\
                                  \end{array}
                                \right),
\]
for some given map $g_j(y).$ Therefore $\psi(x,y)=(x+g(y),y)$, for some map smooth
map $g$. Since clearly, $\psi$ is a diffeomorphism, the result follows.
\end{proof}

\begin{corol}\label{ccc}Let $\alpha$ be a primitive $p$-root of unity.
A map $F:\K^2\to\K^2$,
$$
    F(x,y)=\left(\alpha x+\sum_{k=2}^\infty f_k y^k,\frac y \alpha\right),
$$
is $p$-periodic if and only if $f_k=0$ for all $k=mp-1, m\in\N.$ In
particular, it is always periodic if $f$ is a polynomial and
$p>\deg(f)+1.$
\end{corol}
\begin{proof} Let us write condition~\eqref{ccp} in this setting.  We obtain
\begin{align*}
0=&\sum_{k\ge2}\left(1+\alpha^{k+1}+(\alpha^{k+1})^2+\cdots+
(\alpha^{k+1})^{p-1}\right) f_k y^k\\
=&\sum_{k\ge2,\,k\ne mp-1}\frac{1-(\alpha^{k+1})^p}{1-\alpha^{k+1}}
f_k y^k+ \sum_{k\ge2,\, k=mp-1}p f_k y^k=p\left(\sum_{k\ge2,\,
k=mp-1}f_k y^k\right).
\end{align*}
Therefore, the characterization holds.
\end{proof}

\subsection{Proof of Proposition \ref{cascubic}}

First, we apply Lemma \ref{detconstant} (b), taking into account
that in this case $\mathrm{det}\left(DF(\mathbf{x})\right)\equiv 1$,
obtaining that a necessary condition for a map $F$ in the family of
maps (\ref{primeraP}) to be  periodic, is to belong to one of the
following cases I,II, III and IV considered below. The cases $p\le
4$ follow easily using Corollary~\ref{pbaix}. So, from now one we
also assume that $p\ge5.$

\noindent \textsl{Case I:} It holds that $g_{1,2}\neq 0$,  $g_{0,3}\neq 0$ and
$$\begin{cases}\begin{array}{ll}
f_{2,0}=-\displaystyle{\frac{{\alpha}^{2}g_{1,1}}{2}} ,
f_{1,1}=-\displaystyle{\frac {3{\alpha}^{2}g_{1,1}g_{0,3}}{
g_{1,2}}}, f_{0,2}=-\displaystyle{\frac
{9{\alpha}^{2 }g_{0,3}^{2}g_{1,1}}{2{g_{1,2}}^{2}}},\\
f_{3,0}=-\displaystyle{\frac {{\alpha}^{2}g_{1,2}^{2}}{9g_{0,3}}},
f_{2,1}=-{\alpha}^{2}g_{1,2}, f_{1,2}=-3{\alpha}^{2}g_{0,3},
f_{0,3}=-\displaystyle{\frac {3{\alpha}^{2}g_{0,3}^{2} }{g_{1,2}}},\\
g_{2,0}=\displaystyle{\frac {g_{1,1}g_{1,2}}{6 g_{0,3}}}, g_{0,2}={
\frac { 3g_{1,1}g_{0,3}}{2g_{1,2}}}, g_{3,0}=\displaystyle{\frac
{g_{1,2}^{3}}{27g_{0,3}^{2}}},
 g_{2,1}=\displaystyle{\frac {g_{1,2}^{2}}{3g_{0,3}}}.
\end{array}
\end{cases}$$
 When $g_{1,1}\neq 0$, both conditions
$\mathcal{P}_1(F)=\mathcal{P}_2(F)=0$ in Theorem~\ref{aplgene} give
$$g_{0,3}=-{\frac {2\left( \alpha -1\right)  \left(
{\alpha }^{2}+\alpha+1 \right)g_{1,2}^{2}  }{3\alpha\, \left(
\alpha+1
 \right)  \left( 2{\alpha}^{2}+\alpha+2 \right)g_{1,1}^{2} }}.$$
Observe that  $2{\alpha}^{2}+\alpha+2\ne0$ because $\alpha$ is a root of unity.
Applying again Theorem~\ref{aplgene}, the condition $\mathcal{P}_3(F)=0$ writes as
$$ \displaystyle{\frac {{\alpha}^{5} \left( \alpha-1
\right) ^{2} \left( \alpha+1 \right) ^{2} \left( {\alpha}^{2}+\alpha+1 \right)
^{2} P_8(\alpha)\,g_{1,2}^{2}}{3 \left( 2\,{\alpha}^{2}+\alpha+2 \right)
^{2}}}=0,$$ where
$$P_8(\alpha)=30{\alpha}^{8}+36{\alpha}^{7}+133{\alpha}^{
6}+114{\alpha}^{5}+214{\alpha}^{4}+114{\alpha}^{3}+133{\alpha}
^{2}+36\alpha+30.$$

Since $g_{1,2}\ne0$,     $\alpha\neq -1, \alpha^3\neq 1$, and the roots of $P_8$
have modulus different from~$1$ we obtain that the above equality never holds and
there are no periodic maps in this subfamily when $g_{1,1}\ne0$. When $g_{1,1}=0$,
then $\mathcal{P}_1(F)\ne0$   and the same result holds.

\smallskip

 \noindent \textsl{Case II:}
$$\begin{cases}\begin{array}{ll} f_{2,0}=f_{1,1}=f_{3,0}=f_{2,1}=f_{1,2}=0,\\
g_{2,0}=g_{1,1}=g_{0,2}=g_{3,0}=g_{2,1}=g_{1,2}=g_{0,3}=0,
\end{array}
\end{cases}$$
being $f_{0,2}$ and $f_{0,3}$ free parameters. Hence, in this case $F$ has the
form
\begin{align}\label{aa}
F(x,y)=(\alpha x+f_{0,2}y^2+f_{0,3}y^3,y/\alpha) \end{align}
 and the result
follows from Corollary~\ref{ccc}.

\smallskip

\noindent \textsl{Case III:} It holds that $f_{0,2}\neq 0$,
$$\begin{cases}\begin{array}{ll}
f_{2,0}=\displaystyle{\frac {{\alpha}^{4}{g_{0,2}}^{2}}{ f_{0,2}}},
f_{1,1}=-2{\alpha}^{2} g_{0,2},
f_{3,0}=f_{2,1}=f_{1,2}=f_{0,3}=0,\\
g_{2,0}=\displaystyle{\frac {{\alpha}^{4}{ g_{0,2}}^{3}}{{
f_{0,2}}^{2}}}, g_{1,1}=-\displaystyle{\frac {{\alpha}^{2}{
2g_{0,2}}^{2}}{f_{0,2}}}, g_{0,3}=g_{2,1}=g_{1,2 }=g_{3,0}=0,
\end{array}
\end{cases}$$
being  $g_{0,2}$ a free parameter. By using again condition $\mathcal{P}_1(F)=0$
in Theorem~\ref{aplgene}  we obtain that a necessary condition for  periodicity is
$$
\displaystyle{\frac{2{\alpha}^{6} \left( \alpha+1 \right)
 \left( 2{\alpha}^{2}+\alpha+2 \right)g_{0,2}^{3} }{f_{0,2}}}=0.
$$
So $g_{0,2}=0$ and we are in a subcase of~\eqref{aa}. Hence we are done.

\noindent \textsl{Case IV:}
$$\begin{cases}\begin{array}{ll}
f_{2,0}=f_{1,1}=f_{0,2}=f_{3,0}=f_{2,1}=f_{1,2}=f_{0,3}=0,\\
g_{1,1}=g_{0,2}=g_{2,1}=g_{1,2}=g_{0,3}=0,
\end{array}
\end{cases}$$
being $g_{2,0}$ and $g_{3,0}$ free parameters. This case, is
symmetric with respect the Case II and can be treated analogously.

\subsection{Proof of Proposition \ref{propoultima}}

The proof when $p\le3$ follows easily from Corollary~\ref{pbaix}. So, from now on
we can assume that $p\ge4.$
 From Theorem \ref{aplgene2} we have $\mathcal{P}_1(F)=c$ and
$\mathcal{P}_2(F)=t$. Hence, to obtain a periodic map, these
coefficients must vanish. When $c=t=0$,  $\mathcal{P}_4(F)=-\alpha
m+m-ds$, so another necessary periodicity condition is
$m=ds/(1-\alpha).$ Taking into account the above relation, one gets
$\mathcal{P}_3(F)= \left( \alpha s-2b+2s \right) d$.

Let us assume first, that $d\neq0$. In this case, imposing that
$\mathcal{P}_3(F)$ vanishes we get $b=(\alpha+2)s/2$. In this case
$\mathcal{P}_5(F)=-4\mathcal{P}_6(F)=-4rd^2$, hence $r=0$ is another
necessary periodicity condition. Assuming that $r$ vanishes and
using the expressions of $\mathcal{P}_7(F)$ and $\mathcal{P}_8(F)$
given in Appendix~B we get
$$
\mathcal{P}_7(F)=-{d}^{2}s \left( \alpha-2 \right)  \left( \alpha
s-2d \right)/2\,\mbox{ and }\,\mathcal{P}_8(F)=-{d}^{2}s \left(
\alpha s-2\,d \right)/2.
$$
If $s=0$, the map is then given by $ F(x,y)=(\alpha x+dy^2,y)$ which
is $p$-periodic  if and only if  $\alpha$ is a primitive $p$-root of
the unity, because
$$
F^p(x,y)=\left(\alpha^p
x+dy^2\sum_{i=0}^{p-1}\alpha^i,y\right)=\left( \alpha^p
x+dy^2\left(\frac{1-\alpha^{p}}{1-\alpha}\right),y\right)=\left(
x,y\right).
$$
If $s\neq 0$, then from the above expressions of $\mathcal{P}_7(F)$
and $\mathcal{P}_8(F)$  we have that $d=\alpha s/2$. In this case we
will see that the map is not periodic. Indeed, if
 $c=r=t=0$, $b=(\alpha+2)s/2$, $d=\alpha s/2$ and $m=\alpha
s^2/(2(1-\alpha))$,  then map has a continuum of fixed points
containing the origin, as well as an isolated fixed point at $
\bx_{0}=((1-\alpha)/s,0). $

Notice that if $F$ were $p$-periodic then $\alpha^p=1$ and
$(DF(\bx_{0}))^p=\operatorname{Id}$. Hence $|\det(DF(\bx_{0}))|=1$. On the other
hand $|\det(DF(\bx_0)|=4/|\alpha+1|^2\ne1$, because  $|\alpha|=1$ and recall that
$\alpha\ne1.$ Therefore the map is not periodic.

Finally, when $d=0$, since $m=ds/(1-\alpha),$  the map becomes
$F(x,y)=(\alpha x+bx^2,y+rx^2+sxy)$. From Lemma \ref{detconstant}
its determinant must be constant, which trivially gives $b=s=0$.
Then $ F(x,y)=(\alpha x,y+rx^2)$. Following a similar argument as
above it is easy to see that the map is $p$-periodic  if and only
$\alpha$ is a primitive $p$-root of the unity.


\section{On some second order rational difference
equations}\label{seceed}

\subsection{Global periodicity in the  Lyness
recurrence}\label{seclyn2}

\begin{proof}[{Proof of Proposition~\ref{lynesnormal}}]
It is well known that the dynamics of the  Lyness  recurrence can be studied using
the dynamical system generated by its associated Lyness map
\begin{equation}\label{lynesmap}
G_{a}(x,y)=\left(y,\frac{a+x}{y}\right).
\end{equation}
It is easy to see that $G_{a}^p\ne\operatorname{Id}$ for $p\le4$. So
from now one we search for $p$-periodic maps with $p\ge5.$

In order to significatively simplify the computations to apply
Theorem~\ref{aplgene}  we will introduce a new parameter $\lambda$ being one of
the eigenvalues of the Jacobian matrix of $G_a$ at some fixed point.

Indeed, $G_a$ always has  some fixed point  $(x_0,x_0)$ with $ x_0^2-x_0-a=0$ and
$x_0\ne0$. The eigenvalues $\lambda$ of $DG_a$ at this fixed point satisfy
$\lambda^2-\lambda/x_0+1=0$. Using both equations  it is natural to introduce the
following rational parametrization for $a$,
\begin{equation}\label{barxia}
 a=-\frac{\lambda(\lambda^2-\lambda+1)}{(1+\lambda^2)^2},\quad
 \lambda^2+1\ne0,\,\lambda\ne0,\quad\lambda\in\C.
\end{equation}
Note that this parametrization covers all values of $a\in\C.$ Moreover using it, a
fixed point is $x_0=\lambda/(1+\lambda^2)$ and its associated eigenvalues are
$\lambda$ and $1/\lambda$.

After the translation $(x,y)\to(x-x_0,y-x_0)$, that brings the fixed point to the
origin, the map $G_{a}$ conjugates, using again $x$ and $y$ as variables, with
$$
g_{\lambda}(x,y):=\left(y,-\frac{\lambda x-(1+\lambda^2)y}{\lambda
+x(1+\lambda^2)}\right) \quad\mbox{ with linear part }\quad
L_{\lambda}(x,y)=\left(y,-x+\frac{1+\lambda^2}{\lambda} y\right).
 $$
The linear change of variables $\Psi(x,y)=(x-\lambda y,x-y/\lambda)$ gives a
conjugation between $L_{\lambda}$ and its diagonal form $L(x,y):=(\lambda
x,y/\lambda)$. Using this conjugation we consider the map $F_{\lambda}:=\Psi\circ
g_{\lambda}\circ \Psi^{-1}$, which,  for the sake of brevity, we do not explicitly
write. We will apply Theorem~\ref{aplgene} to $F_\lambda$  which is conjugated to
the Lyness map $G_{a}$.

Imposing that $\mathcal{P}_1(F_\lambda)=0$ we get that the a necessary condition
for $F_\lambda$ to be periodic is
$$
  \left({\lambda}^{2} -\lambda+1 \right)  \left(
{\lambda}^{4 }+{\lambda}^{3}+{\lambda}^{2}+\lambda+1 \right)  \left(
{\lambda}^{2 } +1\right) ^{3} =0.
$$
 The roots of the factor ${\lambda}^{2} -\lambda+1$,  which are primitive
 $6$-th roots of the unity, correspond to the $6$-periodic case,  $a=0$.  Finally, using
again~(\ref{barxia}), we get that all the roots of the polynomial ${\lambda}^{4
}+{\lambda}^{3}+{\lambda}^{2}+\lambda+1$, which are primitive $5$-th roots of the
unity, correspond to the  $5$-periodic case, $a=1$.
\end{proof}

\medskip

\subsection{Non global periodicity in a Gumovski-Mira
recurrence}\label{secgm}

\begin{proof}[Proof of Proposition \ref{gummir}]
Proceeding as in the Lyness case, we consider the map associated with the
Gumovski-Mira recurrence,
$$
G_b(x,y)=\left(y,-x+\frac{y}{ b +y^2}\right).
$$
It is easy to prove that $G_{b}^p\ne\operatorname{Id}$ for $p\le4$.
So from now one we look for $p$-periodic maps with $p\ge5.$

We consider separately the case $b=0.$ In this situation
$\mathbf{x}_0:=(\sqrt2/2,\sqrt2/2)$ is a fixed point of $G_0$. It is
easy to see that $(DG_0(\bx_0)^p\ne\operatorname{Id}$, for any
positive integer $p$, because the matrix is not diagonalizable. So
$G_0$ is not a periodic map.

When $b\ne0$  we introduce a new parameter $\lambda$, and write  $b
=\lambda/(1+\lambda^2)$ with $\lambda^2+1\ne0$ and $\lambda\ne0.$ Notice that this
parametrization covers all values of $b$ in $\C\setminus\{0\}$. We rename  the new
map corresponding to $G_b$ as $g_\lambda$. The eigenvalues of its Jacobian matrix
at the origin, which is always a fixed point, are $\lambda$ and $1/\lambda$. The
linear map $\Psi(x,y)=(x-\lambda y,x-y/\lambda)$ is a conjugation between
$Dg_\lambda(\mathbf{0})$ and its diagonal form $L(x,y):=(\lambda x,y/\lambda)$.
Using this conjugation we consider the map $F_\lambda:=\Psi\circ g_\lambda \circ
\Psi^{-1}$. Using Theorem~\ref{aplgene} we impose that
$\mathcal{P}_2(F_\lambda)=0$. We get that a necessary condition for $F_\lambda$ to
be periodic is
$$ \left({\lambda}^{2}+ 1 \right) ^{2} \left(
{\lambda}^{2}+ \lambda+1 \right) =0.$$ If $\lambda$ is a root of ${\lambda}^{2}+
\lambda+1$,  then it is a primitive $3$rd-root of the unity. Then by
Corollary~\ref{pbaix}, $F_\lambda$ should be globally $3$-periodic. But we have
already discarded this possibility.  So the result follows.
\end{proof}

\subsection{Global periodicity in the $2$-periodic non-autonomous
Lyness recurrence}\label{secGPFba}

In this section we study the problem of the global periodicity of the the sequence
generated by the $2$-periodic Lyness recurrence~\eqref{eq}. The sequence  $\{x_n\}$ given by
this recurrence can be reobtained as
\[
(x_1,x_2)\xrightarrow{G_a}(x_2,x_3)\xrightarrow{G_b}(x_3,x_4)
\xrightarrow{G_a}(x_4,x_5)\xrightarrow{G_b}(x_5,x_6)\xrightarrow{G_a}\cdots
\]
where $G_\alpha(x,y)$, with $\alpha\in\{a,b\}$, is the Lyness map given
in~\eqref{lynesmap}. So the behavior of (\ref{eq}) is given by the dynamical
system generated by the map:
\begin{equation}\label{gab}
G_{b,a}(x,y):=G_b\circ G_a (x,y)=\left(\frac {a+y}{x},\frac {a+b x+ y}{x
y}\right).
\end{equation}

\begin{proof}[Proof of Theorem~\ref{teolyness2p}]
As we have seen it suffices to study the perodicity problem for the
map~\eqref{gab}. It is easy to see that
$G_{a}^p\ne\operatorname{Id}$ for $p=1,2,4$. Moreover it is
$3$-periodic if and only if $a=b=0$. Notice that this case
corresponds to the globally 6-periodic recurrence. We  continue
searching $p$-periodic maps with $p\ge5.$

Following similar ideas that in the previous subsections we
introduce a more suitable rational parametrization of $a$ and $b$. We consider
\begin{equation}\label{abBlambda}
\begin{array}{l}
a=\displaystyle{{\frac {{B}^{3} \left( {\lambda}^{2}+1 \right) +\lambda\, \left(
2\,{B }^{3}-1 \right) }{B \left( \lambda+1 \right)^{2}}}  },\\
b=-B+(B^2-a)^2, \quad\mbox{with}\quad B(\lambda+1)\ne0\quad\mbox{and}\quad \lambda\ne0.
\end{array}
\end{equation}
Using these new parameters we cover all the values of $a$ and $b$ in $\C$. Moreover
the fixed point is $(B,B^2-a),$ where $a$ is given
in~\eqref{abBlambda}, and the eigenvalues of $G_{b,a}$ at this point are $\lambda$
and $1/\lambda.$ After a translation $(x,y)\to(x-B,y-(B^2-a))$, which brings the
fixed point to the origin, the map $G_{b,a}$ conjugates, using again $x$ and $y$
as variables, with
$$
g_{B,\lambda}(x,y)=\left({\frac {y-Bx}{x+B}},-\frac{{B}^{2} \left(
\lambda+1 \right) ^{2}x-B \left({\lambda}^{2}+ \lambda+ 1 \right)
y+\lambda\,xy }{\left( B \left( \lambda+1 \right) ^{2}y+\lambda
\right)  \left( x+B
 \right)
}\right),
$$
with linear part.
$$
L_{B,\lambda}(x,y)=\left({-x+{\frac {y}{B}},-{\frac {B \left(
\lambda+1
 \right) ^{2}x}{\lambda}}+{\frac { \left({\lambda}^{2}+ \lambda+1
 \right) y}{\lambda}}}\right).
 $$
The linear change of variables $\Psi(x,y)=\left(x+y, \left(
\lambda+1 \right) Bx+ \left( 1+\frac{1}{\lambda} \right)By\right)$
gives a conjugation between $L_{B,\lambda}$ and its diagonal form
$L(x,y):=(\lambda x,y/\lambda)$. Using this conjugation we consider
the map
$$
F_{B,\lambda}(x,y):=\Psi\circ g_{B,\lambda}(x,y) \circ \Psi^{-1}(x,y),
$$
which satisfies $ DF_{B,\lambda}(0,0)=\mathrm{diag}\left(\lambda ,
1/\lambda\right)$. For simplicity, we omit its explicit expression. Recall
that $\lambda^p-1\ne0$ for $p=1,2,3$.

By  Theorem~\ref{aplgene}, when  $p\ge5$, from both conditions
$\mathcal{P}_i(F_{B,\lambda})=0, i=1,2,$ we obtain the same
periodicity condition $C_1(B,\lambda)=0$, where
\smallskip

\noindent\!\!$\begin{array}{rl} C_1(B,\lambda):=&
{B}^{6}{\lambda}^{10}+9{B}^{6}{\lambda}^{9}+35{B}^{6}{
\lambda}^{8}+80{B}^{6}{\lambda}^{7}+124{B}^{6}{\lambda}^{6}+2{B}
^{3}{\lambda}^{9}+142{B}^{6}{\lambda}^{5}+8{B}^{3}{\lambda}^{8}\\
{}&+ 124{B}^{6}{\lambda}^{4}+18{B}^{3}{\lambda}^{7}+80{B}^{6}{\lambda
}^{3}+32{B}^{3}{\lambda}^{6}+35{B}^{6}{\lambda}^{2}+40{B}^{3}{
\lambda}^{5}+9{B}^{6}\lambda\\
{}&+32{B}^{3}{\lambda}^{4}+{\lambda}^{7}+
{B}^{6}+18{B}^{3}{\lambda}^{3}+3{\lambda}^{6}+8{B}^{3}{\lambda}^
{2}+2{\lambda}^{5}+2{B}^{3}\lambda+3{\lambda}^{4}+{\lambda}^{3}.
\end{array}
 $
 \smallskip

 Using again Theorem \ref{aplgene},  we obtain
another polynomial restriction
$C_2(B,\lambda):=\mathcal{P}_3(F_{B,\lambda})=0$. The expression of
$C_2(B,\lambda)$ is given in Appendix~C. To study the periodicity of
$F_{B,\lambda}$ it suffices to deal with the two conditions $$
C_1(B,\lambda)=0,\quad C_2(B,\lambda)=0. $$ Computing
$R(\lambda):=\operatorname{Res}(C_1(B,\lambda),C_2(B,\lambda);B)$ we
get
\[
R(\lambda)={\lambda}^{36}  \left( \lambda-1 \right) ^{24} \left(
\lambda+1 \right) ^{72}\left( {\lambda}^{2}+1 \right)^6 \left(
{\lambda}^{2}+\lambda+1 \right) ^{24} S^6(\lambda)\,T^6(\lambda),
\]
where $S(\lambda)={\lambda}^{4}+{
\lambda}^{3}+{\lambda}^{2}+\lambda+1 $ and
$T(\lambda)=3{\lambda}^{4}+
15{\lambda}^{3}+20{\lambda}^{2}+15\lambda+3$. Then, a necessary
condition for $F_{B,\lambda}$ to be $p$-periodic with $p\ge4$ is
that  $\lambda$ is a primitive $p$-th of the unity and that either
$S(\lambda)=0$ or $T(\lambda)=0$. Let us discard the former
possibility.

It turns out that $T$ has two real roots and two complex roots of
modulus one. We have to prove that they are not  roots of the unity.
This can be seen, for instance,  proving that~$T$  is not divisible
by any cyclotomic polynomial. This holds because if it had a
cyclotomic polynomial divisor, its degree should be at most $4$. The
cyclotomic polynomials of degree at most $4$ correspond to $p\in
\{1,2,3,4,5,6,8,10,12\}:=\mathcal{D}_4$.  This is because these are
the cases which correspond to cyclotomic polynomials of degree
$\varphi(p)\leq 4$, being $\varphi$ the Euler's function, see for
instance~\cite{NZ}. Since
$$
\operatorname{Res}(T(\lambda),\lambda^p-1;\lambda)\ne0,\quad\mbox{for}\quad
p\in\mathcal{D}_4,
$$
the result follows.

Finally, when
 $S(\lambda)=0 $ notice that $\lambda$ is a primitive $5$-th root of the unity. So,
 by Corollary~\ref{pbaix} if $F_{B,\lambda}$ is $p$-periodic
 it should be $5$-periodic. Therefore it suffices to study whether  $F_{B,\lambda}^5=\operatorname{Id}$ or not, or
 equivalently whether   $G_{b,a}^5=\operatorname{Id}$.
Computing the numerator of the first component of
$G^5_{b,a}(x,y)-(x,y)$  we get that it writes as $a^4b(1-ab)x+O(2)$,
where as usual $O(m)$ denotes terms of degree at least $m$ in $x$
and $y$. Hence only three possibilities for $G_{b,a}$ to be
$5$-periodic appear: either  $a=0$ or $b=0$ or $ab=1$.

The first two cases can easily discarded. It holds that
$G_{0,a}^5\ne\operatorname{Id}$ and $G_{b,0}^5\ne\operatorname{Id}$.
On the other hand, when $b=1/a, a\ne0$ the numerator of the first
component of $G^5_{1/a,a}(x,y)-(x,y)$
 writes as $-a(a-1)^2(a^2+a+1)^2x^2y+O(3)$. Since this last function has to vanish we get three candidates to be
 $5$-periodic: $a=1$ and $a=(-1\pm i\sqrt{3})/2$ with $b=1/a=\overline a$. It is easy see that all them give rise to
  $5$-periodic maps $G_{b,a}$. The last two correspond to the globally $10$-periodic recurrence.
 \end{proof}

\begin{nota} The characterization of the globally periodic difference equations
treated in this section can also be obtained following the approach
developed in \cite{T} that gives all the periodic QRT-maps. This
result also appears in \cite[p. 165]{D} and \cite{JRV}.
\end{nota}

\section{The  third order Lyness recurrence}\label{r3}

We start proving a general result which will useful for solving
the periodicity problem for the Lyness recurrence.

\begin{propo}\label{propocondr3} Consider the smooth family of maps
$$F(x,y,z)=\left(\alpha x+\sum_{{\mathbf m}} f_{{\mathbf m}}
\bx^{\mathbf m},\beta y+\sum_{{\mathbf m}} g_{{\mathbf m}}
\bx^{\mathbf m},\gamma z+ \sum_{{\mathbf m}} h_{{\mathbf m}}
\bx^{\mathbf m}\right),$$ with ${\mathbf m}\in\{(i,j,k)$ such that
$i+j+k\geq 2\}$, and where
 $\bx^{\mathbf
m}=x^{i}y^{j}z^{k}$.
When $\alpha=\pm1$, $\beta\gamma=1,$ and $\beta\ne1,\gamma\ne1,$
some necessary conditions  for it to be periodic are
$$
f_{3,0,0}^{(2)}=f_{1,1,1}^{(2)}=g_{2,1,0}^{(2)}=g_{0,2,1}^{(2)}=h_{2,0,1}^{(2)}=h_{0,1,2}^{(2)}=0,
$$
where $f_{i,j,k}^{(2)}$ and $g_{i,j,k}^{(2)}$ are the expressions given in the second step
of the normal form procedure described in Section~\ref{norderpc}.
\end{propo}

\begin{proof}
By inspection of the \emph{$3$rd order resonance conditions}, we
observe that when $\alpha=\pm 1$ and $\beta\gamma=1$ there appear
the resonances $\alpha^3-\alpha$,
$\alpha\beta\gamma-\alpha$,$\alpha^2\beta-\beta$,$\beta^2\gamma-\beta$,$\alpha^2\gamma-\gamma$
and $\beta^2\gamma-\gamma$ which are associated to the coefficients
$f_{3,0,0}^{(2)}$, $f_{1,1,1}^{(2)}$, $g_{2,1,0}^{(2)}$,
$g_{0,2,1}^{(2)}$, $h_{2,0,1}^{(2)}$ and $h_{0,2,1}^{(2)}$
respectively. So all them must vanish to have a periodic map.
\end{proof}

\begin{proof}[Proof of Proposition~\ref{lynes3d}]
The dynamics of  the third-order Lyness'  equation can be studied
through the Lyness maps
$$
G_{a}(x,y,z)=\left(y,z,\frac{a+y+z}{x}\right).
$$
It is easy to see that $G_{a}^p\ne\operatorname{Id}$ for $p=1,2$.
We  continue searching  $p$-periodic maps with $p\ge3.$
It has always  some fixed point $(x_0,x_0,x_0)$  with  $x_0^2-2x_0-a=0$ and $x_0\ne0$. Moreover the
 eigenvalues $\lambda$ of the Jacobian matrix at this points are given by the zeroes of
 $-(\lambda+1)(\lambda^2-(1+1/x_0)\lambda+1)=0$. These two equations suggest us to introduce the
 rational parametrization of $a$ as
 \[a=-{\frac {
\lambda\left(2\,{\lambda}^{2} -3\,\lambda+2 \right) }{
 \left( {\lambda}^{2}-\lambda+1 \right) ^{2}}}, \quad\mbox{with}\quad {\lambda}^{2}-\lambda+1 \ne0\quad\mbox{and}\quad \lambda\ne0,
 \]
which covers all values of $a\in\C.$  Then  the fixed point is $(x_0,x_0,x_0)$
with $x_0={\lambda}/(\lambda^2-\lambda+1)$ and the eigenvalues of
 $DG_a$ at this point are $-1,\lambda,1/\lambda$. Notice that since $p\ge3$, we can assume $\lambda\ne1$.
  To apply Proposition~\ref{propocondr3}
we perform the translation $(x,y,z)\to(x-x_0,y-x_0,z-x_0)$, which brings the
fixed point to the origin, obtaining
$$
g_{\lambda}(x,y,z):=\left(y,z,\frac{-\lambda x +\left(
{\lambda}^{2}-\lambda+1 \right) y+  \left( {\lambda}^{2}-\lambda+1
\right) z }{ \left( {\lambda}^{2}-\lambda+1 \right)
x+\lambda}\right),$$
 with linear part
$$L_{\lambda}(x,y)=\left(y,z,-x+{\frac {  \left( {\lambda}^{2}-\lambda+1 \right) y}{\lambda}}+{\frac {  \left( {\lambda}^{2}-\lambda+1 \right)
z}{\lambda}} \right).
 $$

The linear change of variables
$\Psi(x,y,z)=(x+y+{\lambda}^{2}z,-x+\lambda\,y+\lambda\,z,x+{\lambda}^{2}y+z)$
gives a conjugation between $L_{\lambda}$ and its diagonal form
$L(x,y,z):=(-x,\lambda y,z/\lambda)$. Using the conjugation $\Psi$, we finally
obtain a map with diagonal linear part $F_\lambda:=\Psi\circ g_{\lambda}\circ
\Psi^{-1}$, which is under the assumptions of Proposition~\ref{propocondr3}.

Applying this proposition and the Normal
Form Algorithm to $F_\lambda$ we can compute
$g_{2,1,0}^{(2)}$. From the equation
$g_{2,1,0}^{(2)}=0$  we obtain that
$$
\left( {\lambda}^{2}-\lambda+1 \right)^3(\lambda^4+1)=0.
$$
Thus $\lambda$ has to be a primitive $8$-th root of the unity.  All these values of $\lambda$ correspond
to the same value $a=1$, which gives a globally 8-periodic recurrence. So the result follows
\end{proof}

\section*{Acknowledgements} GSD-UAB and CoDALab Groups are supported by the
Government of Catalonia through the SGR program. The first and second authors are
also supported by MCYT through grants MTM2008-03437  and the third author by the
grant DPI2011-25822.


\subsection*{Appendix~A. Expression of  $\boldsymbol{\mathcal{P}_3(F)}$
when $\boldsymbol{\alpha\beta=1}$}

Consider the map (\ref{Fdeteo1}), applying the Normal Form Algorithm
one gets that the periodicity condition associated to
$f_{3,2}^{(4)}$ is given by\newline

 \noindent $\begin{array}{l}
\mathcal{P}_3(F):=f_{1,1}{g_{0,2}}{{g^2_{1,1}}}{\alpha}^{17}+( 2{{
g^2_{1,1}}}f_{1,1}{g_{0,2}}-f_{3,1}{g_{0,2}}-f_{1,1}{g_{1,1}}{
g_{1,2}}-f_{1,1}{g_{0,2}}{g_{2,1}} ) {\alpha}^{16}\end{array}$

\noindent $\begin{array}{l}{} {} + ( {
f_{3,2}}+3g^2_{1,1}f_{1,1}{g_{0,2}}+3{f_{1,1}} {g_{0,2}}{
f_{3,0}}-3f_{1,1}{g_{0,2}}{g_{2,1}}+2{f_{1,1}}f_{2,0}{g_{0,2}}{
g_{1,1}}+2f^2_{1,1}{g_{0,2}}{g_{2,0}}\end{array}$

\noindent $\begin{array}{l}{} {} -2f_{3,1}{g_{0,2}}+2
f_{2,1}f_{2,0}{g_{0,2}}+g^2_{1,1}{f_{1,2}}+2f_{2,2}{
g_{1,1}}+f_{1,1}{g_{2,2}}+2{f_{0,2}}f_{1,1}{g_{1,1}}{
g_{2,0}}-f_{1,1}{g_{1,1}}{g_{1,2}}\end{array}$

\noindent $\begin{array}{l}{} {} +g^2_{1,1}f^2_{1,1}
 ) {\alpha}^{15}+( 3g^2_{1,1}f_{1,1}{g_{0,2}}+6
f_{1,1}{g_{0,2}}{f_{3,0}}-6{f_{0,2}}f_{3,0}{g_{1,1}}-4
f_{1,2}f_{2,0}{g_{1,1}}-3{f_{1,1}}{g_{1,1}}{f_{2,1}}
\end{array}$

\noindent $\begin{array}{l} {} +2f_{1,2}{
g_{2,1}}-3{f_{1,2}}{f_{3,0}}+5g^2_{1,1}{f_{1,2}}+2{{
g^2_{1,1}}}f^2_{1,1}+2{f_{0,2}}f_{1,1}{g_{1,1}}{g_{2,0}}-2
f_{2,2}{f_{2,0}}-2f^2_{1,1}{g_{2,1}}
\end{array}$

\noindent $\begin{array}{l} {} -3f_{3,1}{g_{0,2}}-2{ f_{1,1}}
g^2_{0,2}{g_{2,0}}+3f_{1,1}f_{2,0}{g_{0,2}}{
g_{1,1}}+4f_{2,1}f_{2,0}{g_{0,2}}-3f_{0,2}f_{1,1}{g_{3,0}}-2{
f_{1,2}}{g_{0,2}}{g_{2,0}}\end{array}$

\noindent $\begin{array}{l}{} {} +4f_{2,2}{g_{1,1}}+{f_{3,2}}+2{{
g^3_{1,1}}}{f_{0,2}}+2f_{0,2}{g_{1,1}}{g_{2,1}}-4 f_{1,1}{{
f^2_{2,0}}}{g_{0,2}}-4f_{1,1}{g_{0,2}}{g_{2,1}}+6f^2_{1,1}{
g_{0,2}}{g_{2,0}}\end{array}$

\noindent $\begin{array}{l}{} {}
 -3f_{1,1}{f_{3,1}}-2
f_{1,1}f_{2,0}{g_{1,2}}-2
f_{1,1}f_{1,2}{g_{2,0}}-4f_{4,0}{f_{0,2}}-2g^2_{1,1}f_{0,2}{
f_{2,0}}+2f_{1,1}{g_{2,2}} ) {\alpha}^{14}
\end{array}$

\noindent $\begin{array}{l} {} +( 4f_{1,1}{
f_{2,0}}{f_{2,1}}+4g^2_{1,1}f_{1,1}{g_{0,2}}+11f_{1,1}{
g_{0,2}}{f_{3,0}}-12f_{0,2}f_{3,0}{g_{1,1}}+6f^2_{1,1}{
f_{2,0}}{g_{1,1}}\end{array}$

\noindent $\begin{array}{l}
{}-10f_{1,2}f_{2,0}{g_{1,1}}-8f_{1,1}{g_{1,1}}{
f_{2,1}}+2f_{1,2}{g_{2,1}}-3f_{1,2}{f_{3,0}}+4f_{1,2}{{
f^2_{2,0}}}+3f^3_{1,1}{g_{2,0}}+10g^2_{1,1}{ f_{1,2}}\end{array}$

\noindent $\begin{array}{l}
{}+g^2_{1,1}f^2_{1,1}+6f^2_{1,1}{f_{3,0}}-4
f_{0,2}f_{1,1}{g_{1,1}}{g_{2,0}}-2{f_{2,2}} {f_{2,0}}-7{{
f^2_{1,1}}}{g_{2,1}}+12f_{3,0}f_{2,0}{ f_{0,2}}
\end{array}$

\noindent $\begin{array}{l} {} +4f_{0,2}f_{2,0}{
g_{0,2}}{g_{2,0}}-f_{3,1}{g_{0,2}}-2f_{0,2}{g_{0,2}}{g_{3,0}}-6
f_{1,1}g^2_{0,2}{g_{2,0}}+8{f_{0,2}}f_{1,1}f_{2,0}{g_{2,0}}+6
{f_{2,1}}f_{2,0}{g_{0,2}}\end{array}$

\noindent $\begin{array}{l}{} {}
-6{f_{0,2}}f_{1,1}{g_{3,0}}-4{f_{1,2}}{
g_{0,2}}{g_{2,0}}+6{f_{2,2}}{g_{1,1}}+{f_{3,2}}+6g^3_{1,1}{
f_{0,2}}+3f_{1,1}{g_{0,3}}{g_{2,0}}+6f_{0,2}{g_{1,1}}{g_{2,1}}\end{array}$

\noindent $\begin{array}{l}{} {} +6
f_{0,3}{g_{1,1}}{g_{2,0}}-8f_{1,1}{{f^2_{2,0}}}{g_{0,2}}-3{
f_{1,1}}{g_{0,2}}{g_{2,1}}+5{f_{1,1}} {g_{1,1}}{g_{1,2}}+8f_{0,2}{{
f^2_{2,0}}}{g_{1,1}}+11{{ f^2_{1,1}}}{g_{0,2}}{g_{2,0}}\end{array}$

\noindent $\begin{array}{l}{} {}-5 f_{1,1}{f_{3,1}}-5{f_{1,1}}
f_{2,0}{g_{1,2}}+3f_{1,3}{g_{2,0}}-8
f_{1,1}{f_{1,2}}{g_{2,0}}-4f_{4,0}{f_{0,2}}-12{{
g^2_{1,1}}}f_{0,2}{f_{2,0}}+2f_{0,2}{g_{3,1}}\end{array}$

\noindent $\begin{array}{l}{} {}-4f_{0,2}f_{2,0}{ g_{2,1}}+3
f_{1,1}{g_{2,2}} ) {\alpha}^{13}+( 6f_{1,1}{
f_{2,0}}{f_{2,1}}+9g^2_{1,1}f_{1,1}{g_{0,2}}+7{f_{1,1}}{
g_{0,2}}{f_{3,0}}-24f_{0,2}f_{3,0}{g_{1,1}}
\end{array}$

\noindent $\begin{array}{l} {}+19{{
f^2_{1,1}}}f_{2,0}{g_{1,1}}-16f_{1,2}f_{2,0}{g_{1,1}}-15f_{1,1}{
g_{1,1}}{f_{2,1}}+4f_{1,2}{g_{2,1}}-5f_{1,2}{f_{3,0}}-8f_{0,2}{{
f^3_{2,0}}}+4f_{1,2}{{f^2_{2,0}}}\end{array}$

\noindent $\begin{array}{l}{} {} +12f^3_{1,1}{g_{2,0}}+15
g^2_{1,1}{f_{1,2}}-5g^2_{1,1}f^2_{1,1}+12{{
f^2_{1,1}}}{f_{3,0}}+3f_{0,3}{g_{3,0}}-33f_{0,2}f_{1,1}{g_{1,1}}{
g_{2,0}}-2f_{2,2}{f_{2,0}}
\end{array}$

\noindent $\begin{array}{l} {} -13f^2_{1,1}{g_{2,1}} +12
f_{3,0}f_{2,0}{f_{0,2}}+12{f_{0,2}}f_{2,0}{g_{0,2}}{g_{2,0}}+2
f_{3,1}{g_{0,2}}-4f_{0,2}{g_{0,2}}{g_{3,0}}-8f_{1,1}g^2_{0,2}{g_{2,0}}\end{array}$

\noindent $\begin{array}{l}{}
 {}-6f_{0,2}f_{2,1}{g_{2,0}}-8f_{1,1}g^2_{0,2}{g_{2,0}}-6f_{0,2}f_{2,1}{g_{2,0}}+26f_{0,2}f_{1,1}f_{2,0}{
g_{2,0}}-14f_{1,1}f_{2,0}{g_{0,2}}{g_{1,1}}\end{array}$

\noindent $\begin{array}{l}{} {}+2f_{2,1}f_{2,0}{
g_{0,2}}-15f_{0,2}f_{1,1}{g_{3,0}}-2{f_{1,2}}{g_{0,2}}{g_{2,0}}-8f_{1,1}g^2_{0,2}{g_{2,0}}-6f_{0,2}f_{2,1}{g_{2,0}}+2f_{2,1}f_{2,0}{
g_{0,2}}
\end{array}$

\noindent $\begin{array}{l} {} +26f_{0,2}f_{1,1}f_{2,0}{
g_{2,0}}-14f_{1,1}f_{2,0}{g_{0,2}}{g_{1,1}}-15f_{0,2}f_{1,1}{g_{3,0}}-2{f_{1,2}}{g_{0,2}}{g_{2,0}}+2
f_{2,2}{g_{1,1}}-2{f_{3,2}}\end{array}$

\noindent $\begin{array}{l}{} {} +16g^3_{1,1}{f_{0,2}}+6f_{1,1}{
g_{0,3}}{g_{2,0}}+14{f_{0,2}}{g_{1,1}}{g_{2,1}}+12f_{0,3}{g_{1,1}}{
g_{2,0}}-14f_{1,1}{{f^2_{2,0}}}{g_{0,2}}+3f_{1,1}{g_{0,2}}{
g_{2,1}}\end{array}$

\noindent $\begin{array}{l}{}
{}+9f_{1,1}{g_{1,1}}{g_{1,2}}+20f_{0,2}{{f^2_{2,0}}}{
g_{1,1}}-6f^2_{1,1}{{f^2_{2,0}}}+10f^2_{1,1}{g_{0,2}}
{g_{2,0}}-7f_{1,1}{f_{3,1}}-9f_{1,1}{f_{2,0}}{g_{1,2}}
\end{array}$

\noindent $\begin{array}{l} {} +3f_{1,3}{
g_{2,0}}-19f_{1,1}{f_{1,2}}{g_{2,0}}-6f_{0,3}f_{2,0}{g_{2,0}}-4
f_{4,0}{f_{0,2}}-30g^2_{1,1}f_{0,2}{f_{2,0}}+2f_{0,2}{
g_{3,1}}\end{array}$

\noindent $\begin{array}{l}{}
{}-8{f_{0,2}}{f_{2,0}}{g_{2,1}}+4f_{0,2}{g_{0,2}}{g_{1,1}}{
g_{2,0}}+f_{1,1}{g_{2,2}} ) {\alpha}^{12}+( 8
f_{1,1}f_{2,0}{f_{2,1}}+15g^2_{1,1}f_{1,1}{g_{0,2}}\end{array}$

\noindent $\begin{array}{l}{} {} +2f_{1,1}{
g_{0,2}}{f_{3,0}}-18f_{0,2}f_{3,0}{g_{1,1}}+40{{
f^2_{1,1}}}f_{2,0}{g_{1,1}}-12f_{1,2}f_{2,0}{g_{1,1}}-15f_{1,1}{
g_{1,1}}{f_{2,1}}-2{f_{1,2}}{g_{2,1}}
\end{array}$

\noindent $\begin{array}{l} {}+4{f_{1,2}}{f_{3,0}} -8
f_{0,2}{{f^3_{2,0}}}+6f_{1,2}{{f^2_{2,0}}}+29f^3_{1,1}{
g_{2,0}}+11g^2_{1,1}{f_{1,2}}-16g^2_{1,1}{{
f^2_{1,1}}}+21f^2_{1,1}{f_{3,0}}\end{array}$

\noindent $\begin{array}{l}{} {}+3f_{0,3}{g_{3,0}}-74
f_{0,2}f_{1,1}{g_{1,1}}{g_{2,0}}+4{f_{2,2}}{f_{2,0}}-16{{
f^2_{1,1}}}{g_{2,1}}+16f_{3,0}f_{2,0}{f_{0,2}}+12f_{0,2}f_{2,0}{
g_{0,2}}{g_{2,0}}\end{array}$

\noindent $\begin{array}{l}{}
{}+5{f_{3,1}}{g_{0,2}}-6{f_{0,2}}{g_{0,2}}{
g_{3,0}}-8f_{1,1}g^2_{0,2}{g_{2,0}}-6{f_{0,2}}f_{2,1}{
g_{2,0}}+70f_{0,2}f_{1,1}f_{2,0}{g_{2,0}}\end{array}$

\noindent $\begin{array}{l}{} {}-35f_{1,1}f_{2,0}{g_{0,2}}
{g_{1,1}}-4f_{2,1}f_{2,0}{g_{0,2}}-19f_{0,2}{f_{1,1}}{g_{3,0}}-4{
f_{2,2}}{g_{1,1}}+4{g_{1,2}}{f_{0,2}}{g_{2,0}}-3{f_{3,2}}\end{array}$

\noindent $\begin{array}{l}{} {}+24{{
g^3_{1,1}}}{f_{0,2}}+9f_{1,1}{g_{0,3}}{g_{2,0}}+14f_{0,2}{
g_{1,1}}{g_{2,1}}+18f_{0,3}{g_{1,1}}{g_{2,0}}-8f_{1,1}{{
f^2_{2,0}}}{g_{0,2}}+6{f_{1,1}}{g_{0,2}}{g_{2,1}}\end{array}$

\noindent $\begin{array}{l} {}+10f_{1,1}{
g_{1,1}}{g_{1,2}}+40{f_{0,2}}{{f^2_{2,0}}}{g_{1,1}}-10{{
f^2_{1,1}}}{{f^2_{2,0}}}-3f^2_{1,1}{g_{0,2}}{g_{2,0}}-8
f_{1,1}f_{2,0}{g_{1,2}}+3f_{1,3}{g_{2,0}}\end{array}$

\noindent $\begin{array}{l}{} {}-25f_{1,1}{f_{1,2}}{
g_{2,0}}-12f_{0,3}f_{2,0}{g_{2,0}}+4f_{4,0}{f_{0,2}}-54{{
g^2_{1,1}}}f_{0,2}{f_{2,0}}+2f_{0,2}{g_{3,1}}-12f_{0,2}f_{2,0}{
g_{2,1}}\end{array}$

\noindent
$\begin{array}{l}{}+10f_{0,2}{g_{0,2}}{g_{1,1}}{g_{2,0}}-2{f_{1,1}}{g_{2,2}}
 ) {\alpha}^{11}+( -2f_{1,1}f_{2,0}{f_{2,1}}+22{{
g^2_{1,1}}}f_{1,1}{g_{0,2}}-9f_{1,1}{g_{0,2}}{f_{3,0}}\end{array}$

\noindent $\begin{array}{l} {}-10 f_{0,2}f_{3,0}{g_{1,1}}
+54f^2_{1,1}f_{2,0}{g_{1,1}}+2
f_{1,2}f_{2,0}{g_{1,1}}-8{f_{1,1}}{g_{1,1}}{f_{2,1}}-2f_{1,2}{
g_{2,1}}+5f_{1,2}{f_{3,0}}\end{array}$

\noindent $\begin{array}{l}{} -12f_{0,2}{{f^3_{2,0}}}-6f_{1,2}{{
f^2_{2,0}}}+46f^3_{1,1}{g_{2,0}}+4g^2_{1,1}{
f_{1,2}}-25g^2_{1,1}f^2_{1,1}+13f^2_{1,1}{
f_{3,0}}+3f_{0,3}{g_{3,0}}
\end{array}$

\noindent $\begin{array}{l} {}
-115f_{0,2}f_{1,1}{g_{1,1}}{g_{2,0}}+6
f_{2,2}{f_{2,0}}-9f^2_{1,1}{g_{2,1}}-8f_{3,0}f_{2,0}{
f_{0,2}}+4{f_{0,2}}f_{2,0}{g_{0,2}}{g_{2,0}}+5f_{3,1}{g_{0,2}}\end{array}$

\noindent $\begin{array}{l}{}+2{
f_{0,2}}{g_{0,2}}{g_{3,0}}-6f_{0,2}f_{2,1}{g_{2,0}}+98{f_{0,2}}
f_{1,1}f_{2,0}{g_{2,0}}-52f_{1,1}f_{2,0}{g_{0,2}}{g_{1,1}}-10
f_{2,1}f_{2,0}{g_{0,2}}
\end{array}$

\noindent $\begin{array}{l}{}
-17f_{0,2}f_{1,1}{g_{3,0}}+8f_{1,2}{g_{0,2}}
{g_{2,0}}-10{f_{2,2}}{g_{1,1}}+4{g_{1,2}}{f_{0,2}}{g_{2,0}}-3{
f_{3,2}}+32g^3_{1,1}{f_{0,2}}\end{array}$

\noindent $\begin{array}{l}{}+6f_{1,1}{g_{0,3}}{g_{2,0}}+12{
f_{0,2}}{g_{1,1}}{g_{2,1}}-8f^2_{0,2}g^2_{2,0}+15{
f_{0,3}}{g_{1,1}}{g_{2,0}}+9{f_{1,1}}{g_{0,2}}{g_{2,1}}
\end{array}$

\noindent $\begin{array}{l}{} +5{f_{1,1}}{
g_{1,1}}{g_{1,2}}+40{f_{0,2}}{{f^2_{2,0}}}{g_{1,1}}-16{{
f^2_{1,1}}}{{f^2_{2,0}}}-20f^2_{1,1}{g_{0,2}}{g_{2,0}}+7
f_{1,1}{f_{3,1}}-3{f_{1,1}}f_{2,0}{g_{1,2}}
\end{array}$

\noindent $\begin{array}{l}{} -3f_{1,3}{g_{2,0}}-19
f_{1,1}f_{1,2}{g_{2,0}}-12f_{0,3}f_{2,0}{g_{2,0}}+8{f_{4,0}}{
f_{0,2}}-62g^2_{1,1}f_{0,2}{f_{2,0}}-2f_{0,2}{g_{3,1}}\end{array}$

\noindent $\begin{array}{l}{} -4{
f_{0,2}}f_{2,0}{g_{2,1}}+24f_{0,2}{g_{0,2}}{g_{1,1}}{g_{2,0}}-5{
f_{1,1}}{g_{2,2}}  ) {\alpha}^{10}+( -10f_{1,1}f_{2,0}{ f_{2,1}}+21
g^2_{1,1}f_{1,1}{g_{0,2}}\end{array}$

\noindent $\begin{array}{l}{} -11f_{1,1}{g_{0,2}}{
f_{3,0}}+14f_{0,2}f_{3,0}{g_{1,1}}+53f^2_{1,1}{f_{2,0}} {
g_{1,1}}+16f_{1,2}f_{2,0}{g_{1,1}}+5f_{1,1}{g_{1,1}}{f_{2,1}}\end{array}$

\noindent $\begin{array}{l}{} -8
f_{1,2}{g_{2,1}}+11f_{1,2}{f_{3,0}}+4{f_{0,2}}{{f^3_{2,0}}}-8
f_{1,2}{{f^2_{2,0}}}+53f^3_{1,1}{g_{2,0}}-9g^2_{1,1}{
f_{1,2}}-29g^2_{1,1}f^2_{1,1}
\end{array}$

\noindent $\begin{array}{l}{} +2f^2_{1,1}{
f_{3,0}}-6f_{0,3}{g_{3,0}}-130{f_{0,2}} f_{1,1}{g_{1,1}}{g_{2,0}}+6
f_{2,2}{f_{2,0}}+f^2_{1,1}{g_{2,1}}-16f_{3,0}f_{2,0}{ f_{0,2}}
\end{array}$

\noindent $\begin{array}{l}{}
-24f_{0,2}{f_{2,0}}{g_{0,2}}{g_{2,0}}+2f_{3,1}{g_{0,2}}+6
f_{0,2}{g_{0,2}}{g_{3,0}}+4f_{1,1}g^2_{0,2}{g_{2,0}}+98
f_{0,2}{f_{1,1}}f_{2,0}{g_{2,0}}
\end{array}$

\noindent $\begin{array}{l}{}-56f_{1,1}f_{2,0}{g_{0,2}}{ g_{1,1}}
-10{f_{2,1}}f_{2,0}{g_{0,2}}-3f_{0,2}f_{1,1}{g_{3,0}}+4
f_{1,2}{g_{0,2}}{g_{2,0}}-10f_{2,2}{g_{1,1}}\end{array}$

\noindent $\begin{array}{l}{} {}+4{g_{1,2}}f_{0,2}{
g_{2,0}}+28g^3_{1,1}{f_{0,2}}-4{f_{0,2}}{g_{1,1}}{g_{2,1}}-8
f^2_{0,2}g^2_{2,0}+3{f_{0,3}}{g_{1,1}}{g_{2,0}}+14
f_{1,1}{{f^2_{2,0}}}{g_{0,2}}\end{array}$

\noindent $\begin{array}{l}{}+3f_{1,1}{g_{0,2}}{g_{2,1}}-3
f_{1,1}{g_{1,1}}{g_{1,2}}+28f_{0,2}{{f^2_{2,0}}}{g_{1,1}}-4{{
f^2_{1,1}}}{{f^2_{2,0}}}-38f^2_{1,1}{g_{0,2}}{g_{2,0}}+13
f_{1,1}{f_{3,1}}\end{array}$

\noindent $\begin{array}{l}{}
{}+5f_{1,1}f_{2,0}{g_{1,2}}-6f_{1,3}{g_{2,0}}-f_{1,1}
f_{1,2}{g_{2,0}}+8f_{4,0}{f_{0,2}}-56g^2_{1,1}f_{0,2}{
f_{2,0}}-4f_{0,2}{g_{3,1}}\end{array}$

\noindent $\begin{array}{l}{} +4f_{0,2}f_{2,0}{g_{2,1}}+26f_{0,2}{
g_{0,2}}{g_{1,1}}{g_{2,0}}-5f_{1,1}{g_{2,2}} ) {\alpha}^{9}+ (
-16f_{1,1}f_{2,0}{f_{2,1}}+17g^2_{1,1}f_{1,1}{ g_{0,2}}\end{array}$

\noindent $\begin{array}{l}{}
{}-10f_{1,1}{g_{0,2}}{f_{3,0}}+22{f_{0,2}}f_{3,0}{g_{1,1}}+36
f^2_{1,1}{f_{2,0}}{g_{1,1}}+24{f_{1,2}}f_{2,0}{g_{1,1}}+15
f_{1,1}{g_{1,1}}{f_{2,1}}\end{array}$

\noindent $\begin{array}{l}{}-2f_{1,2}{g_{2,1}}+2f_{1,2}f_{3,0}+8
f_{0,2}{{f^3_{2,0}}}-14f_{1,2}{{f^2_{2,0}}}+45f^3_{1,1}{
g_{2,0}}-13g^2_{1,1}{f_{1,2}}-22g^2_{1,1}{{ f^2_{1,1}}}\end{array}$

\noindent $\begin{array}{l}{}
{}-17f^2_{1,1}{f_{3,0}}-6{f_{0,3}}{g_{3,0}}-103
f_{0,2}f_{1,1}{g_{1,1}}{g_{2,0}}+14f^2_{1,1}{g_{2,1}}-24
f_{3,0}f_{2,0}{f_{0,2}}-f_{3,1}{ g_{0,2}}\end{array}$

\noindent $\begin{array}{l}{}-28f_{0,2}f_{2,0}{g_{0,2}}{g_{2,0}}
+10f_{0,2}{g_{0,2}}{g_{3,0}}+10f_{1,1}g^2_{0,2}{
g_{2,0}}+6f_{0,2}f_{2,1}{g_{2,0}}+58f_{0,2}f_{1,1}f_{2,0}{
g_{2,0}}\end{array}$

\noindent
$\begin{array}{l}{}-41{f_{1,1}}f_{2,0}{g_{0,2}}{g_{1,1}}-4f_{2,1}f_{2,0}{
g_{0,2}}+16f_{0,2}f_{1,1}{g_{3,0}}+6{f_{1,2}}{g_{0,2}}{g_{2,0}}-4
f_{2,2}{g_{1,1}}+3{f_{3,2}}
\end{array}$

\noindent $\begin{array}{l}{}+22g^3_{1,1}{f_{0,2}}-6f_{1,1}{
g_{0,3}}{g_{2,0}}-10f_{0,2}{g_{1,1}}{g_{2,1}}-8f^2_{0,2}{{
g^2_{2,0}}}-9f_{0,3}{g_{1,1}}{g_{2,0}}+16f_{1,1}{{f^2_{2,0}}}{
g_{0,2}}\end{array}$

\noindent
$\begin{array}{l}{}-8f_{1,1}{g_{1,1}}{g_{1,2}}-4f_{0,2}{{f^2_{2,0}}}{
g_{1,1}}+6f^2_{1,1}{{f^2_{2,0}}}-36f^2_{1,1}{g_{0,2}}
{g_{2,0}}+10{f_{1,1}}{f_{3,1}}+10f_{1,1}f_{2,0}{g_{1,2}}\end{array}$

\noindent $\begin{array}{l}{}-6f_{1,3}{
g_{2,0}}+20f_{1,1}f_{1,2}{g_{2,0}}+18f_{0,3}f_{2,0}{g_{2,0}}+4
f_{4,0}{f_{0,2}}-26g^2_{1,1}f_{0,2}{ f_{2,0}}-4f_{0,2}{
g_{3,1}}\end{array}$

\noindent
$\begin{array}{l}{}+16{f_{0,2}}f_{2,0}{g_{2,1}}+36f_{0,2}{g_{0,2}}{g_{1,1}}{
g_{2,0}}-2f_{1,1}{g_{2,2}}  ) {\alpha}^{8}+( -10
f_{1,1}f_{2,0}{f_{2,1}}+8g^2_{1,1}f_{1,1}{g_{0,2}}\end{array}$

\noindent $\begin{array}{l}{} {}-3f_{1,1}{
g_{0,2}}f_{3,0}+24f_{0,2}f_{3,0}{g_{1,1}}+12f^2_{1,1}f_{2,0}{
g_{1,1}}+14{f_{1,2}}f_{2,0}{g_{1,1}}+16f_{1,1}{g_{1,1}}{f_{2,1}}\end{array}$

\noindent $\begin{array}{l}{}-2{
f_{1,2}}{g_{2,1}}-f_{1,2}{f_{3,0}}+16f_{0,2}{{f^3_{2,0}}}-2{
f_{1,2}}{{f^2_{2,0}}}+25f^3_{1,1}{g_{2,0}}-14{{
g^2_{1,1}}}{f_{1,2}}-13g^2_{1,1}f^2_{1,1}\end{array}$

\noindent $\begin{array}{l}{}-21{{
f^2_{1,1}}}{f_{3,0}}-6f_{0,3}{g_{3,0}}-62{f_{0,2}} f_{1,1}{
g_{1,1}}{g_{2,0}}-6f_{2,2}{f_{2,0}}+15f^2_{1,1}{g_{2,1}}-12
f_{3,0}f_{2,0}{f_{0,2}}\end{array}$

\noindent $\begin{array}{l}{} {}
-28f_{0,2}{f_{2,0}}{g_{0,2}}{g_{2,0}}-3
f_{3,1}{g_{0,2}}+6{f_{1,1}}g^2_{0,2}{g_{2,0}}+6
f_{0,2}f_{2,1}{g_{2,0}}-22{f_{1,1}}{f_{2,0}}{g_{0,2}}{g_{1,1}}\end{array}$

\noindent $\begin{array}{l}{} +2
f_{2,1}f_{2,0}{g_{0,2}}+21{f_{0,2}}f_{1,1}{g_{3,0}}-4f_{1,2}{
g_{0,2}}{g_{2,0}}+2f_{2,2}{g_{1,1}}-4{g_{1,2}}f_{0,2}{g_{2,0}}+3{
f_{3,2}}\end{array}$

\noindent $\begin{array}{l}{}
{}+10g^3_{1,1}{f_{0,2}}-9f_{1,1}{g_{0,3}}{g_{2,0}}-18
f_{0,2}{g_{1,1}}{g_{2,1}}-8f^2_{0,2}g^2_{2,0}-21
f_{0,3}{g_{1,1}}{g_{2,0}}+12f_{1,1}{{f^2_{2,0}}}{g_{0,2}}\end{array}$

\noindent $\begin{array}{l}{}-5
f_{1,1}{g_{0,2}}{g_{2,1}}-9f_{1,1}{g_{1,1}}{g_{1,2}}-24f_{0,2}{{
f^2_{2,0}}}{g_{1,1}}+20{f_{1,1}}^{2}{{f^2_{2,0}}}-29{{
f^2_{1,1}}}{g_{0,2}}{g_{2,0}}+f_{1,1}{f_{3,1}}\end{array}$

\noindent $\begin{array}{l}{}+9f_{1,1}f_{2,0}{
g_{1,2}}-3f_{1,3}{g_{2,0}}+26f_{1,1}f_{1,2}{g_{2,0}}+18
f_{0,3}f_{2,0}{g_{2,0}}-4f_{4,0}{f_{0,2}}-4g^2_{1,1}f_{0,2}{
f_{2,0}}\end{array}$

\noindent $\begin{array}{l}{}
-2{f_{0,2}}{g_{3,1}}+12f_{0,2}f_{2,0}{g_{2,1}}+22f_{0,2}{
g_{0,2}}{g_{1,1}}{g_{2,0}}+{f_{1,1}}{g_{2,2}} ) {\alpha}^{7}+ (
2f_{1,1}f_{2,0}{f_{2,1}}\end{array}$

 \noindent $\begin{array}{l}{}
{}+3g^2_{1,1}f_{1,1}{
g_{0,2}}+f_{1,1}{g_{0,2}}{f_{3,0}}+12f_{0,2}f_{3,0}{g_{1,1}}-{{
f^2_{1,1}}}f_{2,0}{g_{1,1}}+4{f_{1,2}}f_{2,0}{g_{1,1}}\end{array}$

\noindent $\begin{array}{l}{}+11f_{1,1}{
g_{1,1}}{f_{2,1}}+4f_{1,2}{g_{2,1}}-7f_{1,2}{f_{3,0}}+8f_{0,2}{{
f^3_{2,0}}}+4f_{1,2}{{f^2_{2,0}}}+8f^3_{1,1}{g_{2,0}}-7{{
g^2_{1,1}}}{f_{1,2}}\end{array}$

\noindent $\begin{array}{l}{} {}-3g^2_{1,1}f^2_{1,1}-16{{
f^2_{1,1}}}{f_{3,0}}+3{f_{0,3}}{g_{3,0}}-14f_{0,2}{f_{1,1}}{
g_{1,1}}{g_{2,0}}-6f_{2,2}{f_{2,0}}+13f^2_{1,1}{g_{2,1}}\end{array}$

\noindent $\begin{array}{l}{}+4{ f_{3,0}}f_{2,0}{f_{0,2}}
-2f_{3,1}{g_{0,2}}-2f_{0,2}{g_{0,2}}{
g_{3,0}}+4f_{1,1}g^2_{0,2}{g_{2,0}}+6{f_{0,2}}f_{2,1}{
g_{2,0}}\end{array}$

\noindent $\begin{array}{l}{}
{}-22{f_{0,2}}f_{1,1}f_{2,0}{g_{2,0}}+6{f_{2,1}}f_{2,0}{g_{0,2}}+20f_{0,2}f_{1,1}{
g_{3,0}}-2f_{1,2}{g_{0,2}}{g_{2,0}}+6f_{2,2}{g_{1,1}}
\end{array}$

\noindent $\begin{array}{l}{}
+2{f_{3,2}}+4g^3_{1,1}{f_{0,2}}-6f_{1,1}{
g_{0,3}}{g_{2,0}}-10f_{0,2}{g_{1,1}}{g_{2,1}}-15f_{0,3}{g_{1,1}}{
g_{2,0}}+2f_{1,1}{{f^2_{2,0}}}{g_{0,2}}
\end{array}$

\noindent $\begin{array}{l}{}-4f_{1,1}f_{2,0}{
g_{0,2}}{g_{1,1}}-4{g_{1,2}} f_{0,2}{g_{2,0}}-3f_{1,1}{g_{0,2}}{
g_{2,1}}-5{f_{1,1}}{g_{1,1}}{g_{1,2}}-28f_{0,2}{{f^2_{2,0}}}{
g_{1,1}}\end{array}$

\noindent $\begin{array}{l}{}
{}+16f^2_{1,1}{{f^2_{2,0}}}-10f^2_{1,1}{
g_{0,2}}{g_{2,0}}-5f_{1,1}{f_{3,1}}+5f_{1,1}{ f_{2,0}}{g_{1,2}}+3{
f_{1,3}}{g_{2,0}}+23f_{1,1}f_{1,2}{g_{2,0}}\end{array}$

\noindent $\begin{array}{l}{} +12f_{0,3}f_{2,0}{
g_{2,0}}-4f_{4,0}{f_{0,2}}+14g^2_{1,1}{f_{0,2}}{f_{2,0}}+2
f_{0,2}{g_{3,1}}+8{f_{0,2}}f_{2,0}{g_{2,1}}+3f_{1,1}{g_{2,2}}\end{array}$

\noindent $\begin{array}{l}{} {}+16{f_{0,2}}{g_{0,2}}{
g_{1,1}}{g_{2,0}} ) {\alpha}^{6} +( 8
f_{1,1}f_{2,0}{f_{2,1}}+2f_{1,1}{g_{0,2}}{f_{3,0}}+2f_{0,2}f_{3,0}{
g_{1,1}}-6f^2_{1,1}f_{2,0}{g_{1,1}}\end{array}$

\noindent $\begin{array}{l}{} -8{f_{1,2}}f_{2,0}{
g_{1,1}}+3f_{1,1}{g_{1,1}}{f_{2,1}}+2{f_{1,2}}{g_{2,1}}-4f_{1,2}{
f_{3,0}}+10f_{1,2}{{f^2_{2,0}}}-3f^3_{1,1}{g_{2,0}}\end{array}$

\noindent $\begin{array}{l}{} -3{{
g^2_{1,1}}}{f_{1,2}}-5f^2_{1,1}{f_{3,0}}+3f_{0,3}{
g_{3,0}}-4f_{2,2}{f_{2,0}}+4f^2_{1,1}{g_{2,1}}+8
f_{3,0}f_{2,0}{f_{0,2}}-f_{3,1}{g_{0,2}}\end{array}$

\noindent $\begin{array}{l}{} +4{f_{0,2}}f_{2,0}{g_{0,2}}{ g_{2,0}}
-4{f_{0,2}}{g_{0,2}}{g_{3,0}}-30
f_{0,2}f_{1,1}f_{2,0}{g_{2,0}}+f_{1,1}f_{2,0}{g_{0,2}}{g_{1,1}}+4
f_{2,1}f_{2,0}{g_{0,2}}\end{array}$

\noindent $\begin{array}{l}{} +6{f_{0,2}}f_{1,1}{g_{3,0}}-4f_{1,2}{
g_{0,2}}{g_{2,0}}+4f_{2,2}{g_{1,1}}-4{g_{1,2}}f_{0,2}{g_{2,0}}-{
f_{3,2}}-3f_{1,1}{g_{0,3}}{g_{2,0}}\end{array}$

\noindent $\begin{array}{l}{}-6f_{0,2}{g_{1,1}}{g_{2,1}}-9
f_{0,3}{g_{1,1}}{g_{2,0}}-4f_{1,1}{{f^2_{2,0}}}{g_{0,2}}-2{
f_{1,1}}{g_{0,2}}{g_{2,1}}-2f_{1,1}{g_{1,1}}{g_{1,2}}+6f^2_{1,1}{{f^2_{2,0}}}\end{array}$

\noindent $\begin{array}{l}{}-16f_{0,2}{{ f^2_{2,0}}}{g_{1,1}} -3{{
f^2_{1,1}}}{g_{0,2}}{g_{2,0}}-8f_{1,1}{f_{3,1}}+3f_{1,3}{
g_{2,0}}+7f_{1,1}{f_{1,2}}{g_{2,0}}-6f_{0,3}f_{2,0}{g_{2,0}}\end{array}$

\noindent $\begin{array}{l}{}-4
f_{4,0}{f_{0,2}}+10g^2_{1,1}f_{0,2}{f_{2,0}}+2f_{0,2}{
g_{3,1}}-4f_{0,2}{f_{2,0}}{g_{2,1}}+2f_{1,1}{g_{2,2}} )
{\alpha}^{5}+ ( 10f_{1,1}f_{2,0}{f_{2,1}}\end{array}$

\noindent $\begin{array}{l}{}+f_{1,1}{g_{0,2}}{
f_{3,0}}-2f_{0,2}f_{3,0}{g_{1,1}}-6f_{1,2}f_{2,0}{g_{1,1}}+2{
f_{1,2}}{g_{2,1}}-f_{1,2}{f_{3,0}}-4f_{0,2}{{f^3_{2,0}}}\end{array}$

\noindent $\begin{array}{l}{}+6{ f_{1,2}}{{f^2_{2,0}}}
-2f^3_{1,1}{g_{2,0}}+g^2_{1,1}{{
f^2_{1,1}}}+3f^2_{1,1}{f_{3,0}}+3f_{0,3}{g_{3,0}}+6
f_{0,2}f_{1,1}{g_{1,1}}{g_{2,0}}+2f_{2,2}{f_{2,0}}\end{array}$

\noindent $\begin{array}{l}{}+f^2_{1,1}{
g_{2,1}}+8f_{3,0}f_{2,0}{f_{0,2}}+8{f_{0,2}}f_{2,0}{g_{0,2}}{
g_{2,0}}-4f_{0,2}{f_{1,1}}f_{2,0}{g_{2,0}}+2f_{1,1}f_{2,0}{g_{0,2}}
{g_{1,1}}\end{array}$

\noindent
$\begin{array}{l}{}+2f_{2,1}f_{2,0}{g_{0,2}}+2f_{0,2}f_{1,1}{g_{3,0}}+2
f_{2,2}{g_{1,1}}-{f_{3,2}}-4f_{1,1}{{f^2_{2,0}}}{g_{0,2}}-4{{
f^2_{1,1}}}{{f^2_{2,0}}}+2f^2_{1,1}{g_{0,2}}{g_{2,0}}\end{array}$

\noindent $\begin{array}{l}{}-3{
f_{1,1}}{f_{3,1}}-f_{1,1}f_{2,0}{g_{1,2}}+3f_{1,3}{
g_{2,0}}+f_{1,1}{f_{1,2}}{g_{2,0}}-6{f_{0,3}} f_{2,0}{g_{2,0}}+6{{
g^2_{1,1}}}f_{0,2}{f_{2,0}}\end{array}$

\noindent $\begin{array}{l}{} {}+2f_{0,2}{g_{3,1}}-4f_{0,2}f_{2,0}{
g_{2,1}}+{f_{1,1}}{g_{2,2}} ) {\alpha}^{4}+( -{{
f^3_{1,1}}}{g_{2,0}}+4{f_{0,2}}{{f^2_{2,0}}}{
g_{1,1}}-f_{1,1}f_{2,0}{g_{1,2}}\end{array}$

\noindent $\begin{array}{l}{}-8f^2_{1,1}{{f^2_{2,0}}}-2
f_{1,1}{{f^2_{2,0}}}{g_{0,2}}-2{f_{0,2}}f_{1,1}{g_{3,0}}+f_{1,2}{
f_{3,0}}-4f_{1,2}{f_{2,0}}{g_{1,1}}+2f_{2,2}{f_{2,0}}\end{array}$

\noindent $\begin{array}{l}{} {}-f_{1,1}{
g_{1,1}}{f_{2,1}}+2f^2_{1,1}{f_{3,0}}+f^2_{1,1}f_{2,0}{
g_{1,1}}-{f_{3,2}}-f^2_{1,1}{g_{2,1}}-6f_{0,3}f_{2,0}{
g_{2,0}}-4f_{0,2}f_{2,0}{g_{2,1}}\end{array}$

\noindent $\begin{array}{l}{}-3f_{1,1}f_{1,2}{g_{2,0}}-2
f_{0,2}f_{3,0}{g_{1,1}}-4f_{0,2}{{f^3_{2,0}}}-f_{1,1}{f_{3,1}}+2
f_{1,1}f_{2,0}{f_{2,1}} ) {\alpha}^{3}+( -2{{
f^2_{1,1}}}{{f^2_{2,0}}}\end{array}$

\noindent $\begin{array}{l}{} {}+4f_{0,2}{{f^2_{2,0}}}{
g_{1,1}}+f_{1,2}{f_{3,0}}+f^3_{1,1}{g_{2,0}}+f^2_{1,1}{
f_{3,0}}+4{f_{0,2}}f_{1,1}f_{2,0}{g_{2,0}}+2f_{2,2}{f_{2,0}}+{
f_{1,1}}{f_{3,1}}\end{array}$

\noindent $\begin{array}{l}{}
-2f_{1,2}{f^2_{2,0}}+2f^2_{1,1}f_{2,0}{ g_{1,1}} ) {\alpha}^{2}+ (
-2f_{1,2}{{f^2_{2,0}}}-2 f_{1,1}f_{2,0}{f_{2,1}}-f^2_{1,1}f_{3,0} )
\alpha+2{{ f^2_{1,1}}}{{f^2_{2,0}}}.

\end{array}$

\subsection*{Appendix~B. Expression of  $\boldsymbol{\mathcal{P}_7(F)}$  and
$\boldsymbol{\mathcal{P}_8(F)}$ when $\boldsymbol{\beta=1}$ }  Applying the Normal
Form Algorithm to the map~(\ref{Fdeteo2}) we get that, when $\alpha^3\ne1$, then
$f_{1,4}^{(4)}=\mathcal{P}_7(F)/(\alpha^3-1)$ and
$g_{0,5}^{(4)}=\mathcal{P}_8(F)/(\alpha^3-1)$ where

\noindent $
\begin{array}{l}
\mathcal{P}_7(F):=f_{1,4} {\alpha}^{3}+ ( 3 f_{0,3} g_{1,2}-3
f_{1,4}-2  f_{0,3} f_{2,1}-2 f_{2,0} f_{0,4}+2 g_{1,3} f_{0,2} -2
f_{2,2} f_{0,2}\\
\phantom{x}+4 f_{0,4} g_{1,1} ) {\alpha}^{2}+ ( -4 g_{1,3} f_{0,2}+4
f_{0,3} f_{2,1}-8 f_{0,4}
 g_{1,1}-6 f_{0,3} g_{1,2}-4 g_{2,1}
f^2_{0,2}\\
\phantom{x}-10  f_{0,3} f_{0,2} g_{2,0}+3 f_{1,4}+3 f_{3,0}
f^2_{0,2}+4 f_{2,0} f_{0,4}+4 f_{2,2} f_{0,2} ) \alpha-f_{1,4}-2
f_{2,2}
f_{0,2}\\
\phantom{x}+2 g_{1,3} f_{0,2}+4 g_{2,1} f^2_{0,2}-2 f_{2,0}
f_{0,4}+2 {f_{2,0}}^{2} {f_{0,2} }^{2}-3 f_{3,0} f^2_{0,2}-8 g_{1,1}
f_{2,0}
f^2_{0,2}\\
\phantom{x}+8 f^2_{0,2} g^2_{1,1}+4 f_{0,4} g_{1,1}+3 f_{0,3}
g_{1,2}+10 f_{0,3} f_{0,2} g_{2,0}-2 f_{0,3}  f_{2,1},\\
\mathcal{P}_8(F):= g_{0,5} {\alpha}^{3}+ ( -3 g_{0,5}-f_{0,4}
g_{1,1}-g_{1,3} f_{0,2}-f_{0,3} g_{1,2} ) {\alpha}^{2}+ ( 2 f_{0,3}
g_{1,2}+2 f_{0,4} g_{1,1}\\
\phantom{x}+g_{2,1} f^2_{0,2}+2 f_{0,3} f_{0,2} g_{2,0}+2 g_{1,3}
f_{0,2}+3
g_{0,5} ) \alpha-f_{0,3} g_{1,2}-g_{0,5}-2 f_{0,3} f_{0,2} g_{2,0}\\
\phantom{x}- f_{0,4} g_{1,1}+g_{1,1} f_{2,0} f^2_{0,2}- g_{1,3}
f_{0,2}-g_{2,1} f^2_{0,2}-2 f^2_{0,2}g^2_{1,1}.
\end{array}
$

\subsection*{Appendix~C. Expression of  $\boldsymbol{C_2(B,\lambda)}$ in the proof
of Theorem~\ref{teolyness2p}}

\noindent $
\begin{array}{ll}
 C_2(B,\lambda):=& 8{B}^{12}{\lambda}^{25}+125{B}^{12}{\lambda}^
{24}+912{B}^{12}{\lambda}^{23}+4140{B}^{12}{\lambda}^{22}+13091{
B}^{12}{\lambda}^{21}+23{B}^{9}{\lambda}^{24}\end{array}$

\noindent $\begin{array}{rl}{} {}&+30388{B}^{12}{
\lambda}^{20}+264{B}^{9}{\lambda}^{23}+52493{B}^{12}{\lambda}^{19}
+1457{B}^{9}{\lambda}^{22}+64792{B}^{12}{\lambda}^{18}\end{array}$

\noindent $\begin{array}{rl}{} {}&+5130{B}^{
9}{\lambda}^{21}+44963{B}^{12}{\lambda}^{17}+12792{B}^{9}{\lambda}
^{20}+17{B}^{6}{\lambda}^{23}-22114{B}^{12}{\lambda}^{16}\end{array}$

\noindent $\begin{array}{rl}{} {}&+23399{
B}^{9}{\lambda}^{19}+152{B}^{6}{\lambda}^{22}-126694{B}^{12}{
\lambda}^{15}+30518{B}^{9}{\lambda}^{18}+685{B}^{6}{\lambda}^{21}\end{array}$

\noindent $\begin{array}{rl}{} {}&-
230443{B}^{12}{\lambda}^{14}+23012{B}^{9}{\lambda}^{17}+2027{B}^
{6}{\lambda}^{20}-285544{B}^{12}{\lambda}^{13}-7945{B}^{9}{\lambda
}^{16}\end{array}$

\noindent $\begin{array}{rl}{}
{}&+4241{B}^{6}{\lambda}^{19}-265465{B}^{12}{\lambda}^{12}-
59005{B}^{9}{\lambda}^{15}+6222{B}^{6}{\lambda}^{18}+23{B}^{3}{
\lambda}^{21}\end{array}$

\noindent $\begin{array}{rl}{}
{}&-182980{B}^{12}{\lambda}^{11}-111409{B}^{9}{\lambda}^
{14}+5530{B}^{6}{\lambda}^{17}+126{B}^{3}{\lambda}^{20}-80299{B}
^{12}{\lambda}^{10}\end{array}$

\noindent $\begin{array}{rl}{}
{}&-140407{B}^{9}{\lambda}^{13}-138{B}^{6}{\lambda
}^{16}+356{B}^{3}{\lambda}^{19}-280{B}^{12}{\lambda}^{9}-131599{
B}^{9}{\lambda}^{12}\end{array}$

\noindent $\begin{array}{rl}{}
{}&-10552{B}^{6}{\lambda}^{15}+644{B}^{3}{\lambda
}^{18}+37544{B}^{12}{\lambda}^{8}-90967{B}^{9}{\lambda}^{11}-21809
{B}^{6}{\lambda}^{14}\end{array}$

\noindent $\begin{array}{rl}{}
{}&+723{B}^{3}{\lambda}^{17}+40086{B}^{12}{
\lambda}^{7}-40111{B}^{9}{\lambda}^{10}-28180{B}^{6}{\lambda}^{13}
+253{B}^{3}{\lambda}^{16}+8{\lambda}^{19}\end{array}$

\noindent $\begin{array}{rl}{} {}&+26571{B}^{12}{\lambda}
^{6}-883{B}^{9}{\lambda}^{9}-26229{B}^{6}{\lambda}^{12}-844{B}^{
3}{\lambda}^{15}+29{\lambda}^{18}+12701{B}^{12}{\lambda}^{5}\end{array}$

\noindent $\begin{array}{rl}{} {}&+17318
{B}^{9}{\lambda}^{8}-17474{B}^{6}{\lambda}^{11}-2101{B}^{3}{
\lambda}^{14}+36{\lambda}^{17}+4481{B}^{12}{\lambda}^{4}+18449{B
}^{9}{\lambda}^{7}\end{array}$

\noindent $\begin{array}{rl}{} {}&
-6870{B}^{6}{\lambda}^{10}-2804{B}^{3}{\lambda}^
{13}+34{\lambda}^{16}+1143{B}^{12}{\lambda}^{3}+12036{B}^{9}{
\lambda}^{6}+902{B}^{6}{\lambda}^{9}\end{array}$

\noindent $\begin{array}{rl}{}&-2563{B}^{3}{\lambda}^{12}-33
{\lambda}^{15}+199{B}^{12}{\lambda}^{2}+5561{B}^{9}{\lambda}^{5}
+4012{B}^{6}{\lambda}^{8}-1532{B}^{3}{\lambda}^{11}-71{\lambda}^
{14}\end{array}$

\noindent $\begin{array}{rl}{}
&+21{B}^{12}\lambda+1827{B}^{9}{\lambda}^{4}+3635{B}^{6}{
\lambda}^{7}-364{B}^{3}{\lambda}^{10}-137{\lambda}^{13}+{B}^{12}+
404{B}^{9}{\lambda}^{3}\end{array}$

\noindent $\begin{array}{rl}{}
&+2079{B}^{6}{\lambda}^{6}+335{B}^{3}{
\lambda}^{9}-92{\lambda}^{12}+53{B}^{9}{\lambda}^{2}+839{B}^{6}{
\lambda}^{5}+493{B}^{3}{\lambda}^{8}-56{\lambda}^{11}\end{array}$

\noindent $\begin{array}{rl}{} &+3{B}^{9}
\lambda+232{B}^{6}{\lambda}^{4}+348{B}^{3}{\lambda}^{7}+8{
\lambda}^{10}+37{B}^{6}{\lambda}^{3}+149{B}^{3}{\lambda}^{6}+29{
\lambda}^{9}+2{B}^{6}{\lambda}^{2}\end{array}$

\noindent $\begin{array}{rl}{} &+35{B}^{3}{\lambda}^{5}+25{
\lambda}^{8}+3{B}^{3}{\lambda}^{4}+9{\lambda}^{7}+{\lambda}^{6}.
\end{array}
 $

\end{document}